\documentclass[a4paper,10pt]{amsart}
\usepackage{amsmath}
\usepackage{amsfonts}
\usepackage{amssymb}
\usepackage{amsthm}
\usepackage{graphics}
\usepackage{verbatim}

\numberwithin{equation}{section} 

\allowdisplaybreaks

\newtheorem{theorem}{Theorem} 
\newtheorem{lemma}[theorem]{Lemma}

\newtheorem{proposition}[theorem]{Proposition}
\theoremstyle{definition}

\newcommand{\ADS}{\mathrm{AdS}}
\newcommand{\EIN}{\mathrm{Ein}}
\newcommand{\ADSL}{\Delta_{n-1}}
\newcommand{\LAP}{\Delta}
\newcommand{\Ad}{\mathrm{Ad}}
\newcommand{\FRAKB}{\mathfrak{q}}

\newcommand{\PHAT}{\hat{P}}

\def\c{{\mathbb C}}   

\newcommand{\ARBIT}{\mathcal{U}(\mathfrak{n}_{n+1}^{-})
(\mathfrak{m}_{n+1} \oplus \mathbb{C}(H-\lambda))}

\newcommand{\Ind}{\mathrm{Ind}}

\usepackage[latin1]{inputenc}

\begin{document}

\author[]{Pierre B\"acklund} \title[Equivariant differential operators
and AdS spaces]{Families of equivariant differential operators and
anti de Sitter spaces} \date{\today} 

\address{Matematiska institutionen, Uppsala universitet, Box 480, SE-75106
Uppsala, SWEDEN}
\email{pierre@math.uu.se}

\maketitle

\tableofcontents

\section{Introduction}\label{intro}

Let $\widehat{\EIN}_n$ be the compact Lorentzian manifold $S^1 \times
S^{n-1}$ with the metric $(-g_{S^1}) \oplus g_{S^{n-1}}$. The
conformal group of $\widehat{\EIN}_n$, i.e., the group of all
diffeomorphisms which induce conformally equivalent metrics, can be
identified with $O(2,n)$. Let
$$
i: \widehat{\EIN}_{n-1} \hookrightarrow \widehat{\EIN}_n
$$ 
be the isometric embedding which is induced by the equatorial
embedding $S^{n-2} \to S^{n-1}$ of spheres. We construct a sequence
$\hat{D}_N^{\text{c}}(\lambda)$, $N \ge 0$ of polynomial families
$$
\hat{D}_N^{\text{c}}(\lambda): C^\infty(\widehat{\EIN}_n) \to
C^\infty(\widehat{\EIN}_{n-1})
$$ 
of $O(2,n-1)$-equivariant differential operators. Here equivariance
is understood with respect to respective spherical principal series
representations $\pi_\lambda^c$ on $C^\infty(\widehat{\EIN}_n)$ and
$C^\infty(\widehat{\EIN}_{n-1})$.

The results extend corresponding results of \cite{juhl_conform} on
$SO(1,n)^\circ$-equivariant families $D_N^c(\lambda): C^\infty(S^n)
\to C^\infty(S^{n-1})$. In \cite{juhl_conform} the families
$D_N^c(\lambda)$ serve as conformally flat models of corresponding
conformally covariant ``curved analogs'' $C^\infty(M) \to
C^\infty(\Sigma)$ which are canonically associated to any hypersurface
$\Sigma$ in a Riemannian manifold $M$. In \cite{juhl_conform} it is
shown how such families contain information on Branson's
$Q$-curvature. In particular, combining the holographic formula for
$Q$-curvature (\cite{juhl_holo}) with the structure of the
intertwining families, provides recursive relations for $Q$-curvature.

In the present situation, the conformal action of $O(2,n)$ on
$\widehat{\EIN}_n$ replaces the conformal action of $O(1,n)$ on
$S^{n-1}$. Whereas $S^{n-1}$ is the boundary at infinity of the
hyperbolic space $\mathbb{H}^n$ with the isometry group $O(1,n)$, here
$\widehat{\EIN}_n$ appears as the conformal boundary of the
$n+1$-dimensional anti de Sitter space $\ADS_{n+1}$ with the isometry
group $O(2,n)$ (see Section \ref{prel}). The universal covering space of
$\widehat{\EIN}_n$ is given by
$$
(\mathbb{R} \times S^{n-1},-dr^2+g_{S^{n-1}})
$$ 
and is also known as Einstein's static universe; \cite{frances_conformal}.

For even $N$, the family $\hat{D}_N^{\text{c}}(\lambda)$ drops down to
a polynomial family
\begin{equation}
D_N^{\text{c}}(\lambda) : C^{\infty}(\EIN_{n}) \rightarrow
C^{\infty}(\EIN_{n-1})
\label{DROP}
\end{equation}
of $O(2,n-1)$-equivariant differential operators, where
$$ 
\EIN_n = (S^{1} \times S^{n-1})/\mathbb{Z}_2,
$$ 
and the $\mathbb{Z}_2$-action is coming from the involution
$S^1 \times S^{n-1} \rightarrow S^1 \times S^{n-1}$ given by
$(x,y) \mapsto (-x,-y)$. 

For odd $N$, the family
$\hat{D}^{\text{c}}_{N}(\lambda)$ does not drop down to a well-defined
$O(2,n-1)$-equivariant map
$C^{\infty}(\EIN_{n}) \rightarrow C^{\infty}(\EIN_{n-1})$. The
non-existence of odd-order families has a geometric reason. In \cite{juhl_conform} it is shown that for any (orientable) hypersurface $\Sigma$
of an (orientable) manifold $M$, the family
$$
D_1(\lambda) = i^* \nabla_N + \lambda H i^*: C^\infty(M) \to
C^\infty(\Sigma)
$$
is conformally covariant. Here $N$ denotes a unit normal field on
$\Sigma$ and $H$ is the corresponding mean curvature. $\widehat{\textrm{Ein}}_n$ is an (orientable) hypersurface in $\widehat{\mathrm{Ein}}_{n+1}$ with
$H=0$,
and the family $D_1(\lambda)$ reduces to the equivariant family $D_1^c = i^*
\nabla_N$
given by the normal derivative. Now for odd $n$, $\text{Ein}_n$ is a non-orientable codimension one submanifold of
the
orientable manifold $\text{Ein}_{n+1}$, and the above construction does
not descent to $\text{Ein}_n \hookrightarrow \text{Ein}_{n+1}$. Similarly, for even $n$,
$\text{Ein}_{n+1}$ is non-orientable, and the above construction does not
descent.

The space $\widehat{\EIN}_{n-1}$ can alternatively be described as the
set of rays through the origin in the light cone
\begin{equation}
C_n = \left\{ -t_1^2 - t_2^2 + x_1^2 + \dots + x_{n-1}^2 = 0 \right\}.
\label{lightcone}
\end{equation}
$G^n = O(2,n-1)$ acts transitively on
$\widehat{\mathrm{Ein}}_{n-1}$, and we let $\hat{P}^n \subset G^n$ be
the isotropy group of the ray generated by $(1,0,0,1,0,\dots,0) \in
\mathbb{R}^{n+1}$. In a similar way, $\EIN_{n-1}$ is given by the set
of lines in $C_n$ through the origin, and we let $P^n \subset G^n$
denote the isotropy group of the line through $(1,0,0,1,0,\dots,0) \in
\mathbb{R}^{n+1}$.

As in \cite{juhl_conform} we discuss two different types of
constructions and prove that both lead to the same result. One
construction is in terms of representation theory and one is in terms
of spectral theory of an associated Laplacian.


We start with the description of the Lie-theoretic construction. For
$n \geq 4$ let $G^n = O(2,n-1)$ with Lie algebra $\mathfrak{g}_n =
\mathfrak{o}(2,n-1)$. The isotropy group $P^n$ is a maximal parabolic
subgroup with Langlands decomposition $P^n = M^n A (N^+)^n$. The Lie
algebra $\mathfrak{p}_n$ of $P^n$ has the Langlands decomposition
$\mathfrak{p}_n= \mathfrak{m}_n \oplus \mathfrak{a} \oplus
\mathfrak{n}_n^+$ with a one-dimensional space $\mathfrak{a}$. Here
$\mathfrak{a}$ is spanned by $H$ and the Lie subalgebras
$\mathfrak{n}_n^\pm$ are the respective $\pm$-eigenspaces of
$\text{ad}(H)$. Let $(N^-)^n = \exp(\mathfrak{n}_n^-)$.

For $\lambda \in \c$ we define the character $\xi_{\lambda}:
\mathfrak{p}_n \rightarrow \c = \c_{\lambda}$ by $\xi_{\lambda}(tH) =
t\lambda$ ($\xi_{\lambda}$ acts by $0$ on $\mathfrak{m}_n \oplus
\mathfrak{n}_n^+$). We shall use the same notation $\xi_{\lambda}$ for
the character of $P$ given by $\xi_{\lambda}(\exp(tH))= e^{\lambda t}$.

Let $\mathcal{U}(\mathfrak{g}_n)$ be the universal enveloping algebra
of $\mathfrak{g}_n$. For $\lambda \in \c$ the character $\xi_\lambda$
of $\mathfrak{p}_n$ gives rise to the generalized Verma module (see
\eqref{GVM})
$$
\mathcal{M}_{\lambda}(\mathfrak{g}_n) = \mathcal{U}(\mathfrak{g}_n)
\otimes_{\mathcal{U}(\mathfrak{p}_n)} \c_{\lambda}.
$$ 
Similarly, for the character $\xi_\lambda$ of $P$ we consider the
induced representation $\Ind_P^G (\xi_{\lambda})$ of $G$.

Let $i:\mathcal{U}(\mathfrak{g}_n) \rightarrow
\mathcal{U}(\mathfrak{g}_{n+1})$ be the map induced by the inclusion
$\mathfrak{g}_n \subset \mathfrak{g}_{n+1}$ given by \eqref{inkl}.

\begin{theorem}\label{I-1} For any non-negative integer $N$ there 
exists a polynomial family $\mathcal{D}^{0}_N(\lambda) \in
\mathcal{U}(\mathfrak{n}_{n+1}^-)$ such that the map
$$ 
\mathcal{U}(\mathfrak{g}_n) \otimes \c_{\lambda-N} \ni T
\otimes 1 \mapsto i(T) \mathcal{D}^{0}_N(\lambda) \otimes 1 \in
\mathcal{U}(\mathfrak{g}_{n+1}) \otimes \c_{\lambda}
$$ 
induces a homomorphism
$$ 
\mathcal{M}_{\lambda-N}(\mathfrak{g}_n) \rightarrow
\mathcal{M}_{\lambda}(\mathfrak{g}_{n+1})
$$ of $\mathcal{U}(\mathfrak{g}_n)$-modules for all
$\lambda$. Furthermore, for any $\lambda \in \c$,
$\mathcal{D}^{0}_N(\lambda)$ spans the space of all homomorphisms
$\mathcal{M}_{\lambda-N}(\mathfrak{g}_n) \rightarrow
\mathcal{M}_{\lambda}(\mathfrak{g}_{n+1})$ of
$\mathcal{U}(\mathfrak{g}_n)$-modules.
\end{theorem}

In addition to the existence, we will actually show how to find
explicit formulas for the intertwining families
$\mathcal{D}^{0}_N(\lambda)$. Now viewing
$\mathcal{D}^{0}_{N}(\lambda) \in \mathcal{U}(\mathfrak{n}_{n+1}^-)$
as a left-invariant differential operator acting on
$C^{\infty}(G^{n+1})$ from the right, we obtain the following result.

\begin{theorem}\label{I-2} (a) For any non-negative integer $N$ the 
polynomial family \newline $\mathcal{D}^{0}_{N}(\lambda) \in
\mathcal{U}(\mathfrak{n}_{n+1}^-)$ induces a family of left
$G^{n+1}$-equivariant operators
$$ 
D'_N(\lambda): \Ind_{P^{n+1}}^{G^{n+1}}(\xi_{\lambda}) \rightarrow
\Ind_{P^n}^{G^{n+1}}(\xi_{\lambda-N} \otimes \sigma),
$$ 
where for even $N$, $\sigma$ is the trivial representation of $M$,
and for odd $N$, $\sigma = \sigma_-$ (see section \ref{principal}).
The composition $D_N(\lambda) = i^* \circ D'_N(\lambda)$ of
$D'_N(\lambda)$ with the restriction map $i^*: C^\infty(G^{n+1})
\rightarrow C^\infty(G^n)$ defines a family of left $G^n$-equivariant
operators
\begin{equation}\label{D-even}
D_N(\lambda): \Ind_{P^{n+1}}^{G^{n+1}}(\xi_{\lambda}) \rightarrow
\Ind_{P^n}^{G^{n}}(\xi_{\lambda-N} \otimes \sigma).
\end{equation} \\
(b) For any non-negative integer $N$ the polynomial family
$\mathcal{D}^{0}_{N}(\lambda) \in \mathcal{U}(\mathfrak{n}_{n+1}^-)$
induces a family of left $G^{n+1}$-equivariant operators
$$
\hat{D}^{'}_N(\lambda): \Ind_{\PHAT^{n+1}}^{G^{n+1}}(\xi_{\lambda}) \rightarrow
\Ind_{\PHAT^n}^{G^{n+1}}(\xi_{\lambda-N}).
$$ 

The composition $\hat{D}_N(\lambda) = i^* \circ \hat{D}'_N(\lambda)$
of $\hat{D}'_N(\lambda)$ with the restriction map $i^*:
C^{\infty}(G^{n+1}) \rightarrow C^{\infty}(G^n)$ defines a a family of
left $G^n$-equivariant operators
\begin{equation}\label{D-all2}
\hat{D}_N(\lambda): \Ind_{\PHAT^{n+1}}^{G^{n+1}}(\xi_{\lambda}) \rightarrow
\Ind_{\PHAT^n}^{G^{n}}(\xi_{\lambda-N}).
\end{equation}
\end{theorem}

Using
$$ 
\EIN_{m-1} \simeq G^m/P^m \quad \mbox{and} \quad
\widehat{\EIN}_{m-1} \simeq G^m/\hat{P}^m,
$$ 
the representation spaces of spherical principal series can be identified with
$$
C^{\infty}(\EIN_{m-1}) \quad \mbox{and} \quad C^{\infty}(\widehat{\EIN}_{m-1}),
$$ respectively. The notation $C^{\infty}(\EIN_m)_\lambda$ and
$C^{\infty}(\widehat{\EIN}_m)_\lambda$ emphasizes the respective
module structure. In these terms, Theorem \ref{I-1} admits the
following interpretation.

\begin{theorem}\label{I-3} (a) For any even integer $N \geq 0$ the polynomial 
family \eqref{D-even} defines a family of $G^n$-equivariant operators
$$ 
D_N^{\mathrm{c}}(\lambda): C^{\infty}(\EIN_n)_\lambda \rightarrow
C^{\infty}(\EIN_{n-1})_{\lambda-N}. 
$$ 
(b) For any non-negative integer $N$ the polynomial family
\eqref{D-all2} defines a family of $G^n$-equivariant operators
$$
\hat{D}_N^{\mathrm{c}}(\lambda): C^{\infty}(\widehat{\EIN}_n)_\lambda
\rightarrow C^{\infty}(\widehat{\EIN}_{n-1})_{\lambda-N}.
$$
\end{theorem}

Now let $\mathbb{M}^{n}$ be the Minkowski space with the Lorentzian
metric
\begin{equation}
ds^2 = -dx_1^2 + dx_2^2 + \dots + dx_{n}^2.
\label{minkowskimetric}
\end{equation}
The families $D_N^c(\lambda)$ in Theorem \ref{I-3} give rise to
families
\begin{equation}
D_N^{\text{nc}}(\lambda): C^{\infty}(\mathbb{M}^n) \rightarrow
C^{\infty}(\mathbb{M}^{n-1})
\end{equation}
as follows. $\EIN_m$ is conformally flat. More precisely, the
composition of the inclusion $(N^-)^m \rightarrow G^m$ with the
projection $G^m \rightarrow G^m/P^m$ defines a conformal
embedding $j: \mathbb{M}^{m-1} \to \EIN_{m-1}$. Here we identify
$\mathbb{M}^{m-1} \simeq (N^-)^m$ (see \eqref{IDENT}) and
$\text{Ein}_{m-1} \simeq G^m/P^m$. The embedding $j$ yields a
well-known conformal compactification of Minkowski space; see also \cite{barbot}. Now the
families $D_N^{\text{nc}}(\lambda)$ are defined by restriction to the
open sets $j(\mathbb{M}^m) \subset \EIN_m$.

The family $D_{2N}^{\text{nc}}(\lambda)$ is a polynomial in the
Laplacian $\Delta_{\mathbb{M}^{n-1}}$ and $\partial^2/\partial x_n^2$
 with polynomial coefficients in $\lambda$. In particular,
$$ 
D_{2N}^{\text{nc}}\left(-\frac{n}{2}\!+\!N\right) = i^*
\Delta_{\mathbb{M}^{n}}^N \quad \mbox{and} \quad
D_{2N}^{\text{nc}}\left(-\frac{n\!-\!1}{2}\!+\!N\right) =
\Delta_{\mathbb{M}^{n-1}}^N i^*.
$$ 
where $i$ is the inclusion map $\mathbb{M}^{n-1} \hookrightarrow \mathbb{M}^n$.

The powers $(\Delta_{\mathbb{M}^m})^k$ of the Laplacian
$\Delta_{\mathbb{M}^m}$ on flat Minkowski space are very special
cases of the conformally covariant powers of the Laplacian of
pseudo-Riemannian manifolds constructed in \cite{GJMS}. Their
$O(2,m)$-equivariance is a consequence of their conformal covariance; see
also \cite{kostant}, \cite{jakobsen_vergne}.

Now we turn to an alternative construction of the families
$D_N^{\text{nc}}(\lambda)$ in terms of the asymptotics of
eigenfunctions of the Laplacian on anti de Sitter space $\text{AdS}_n$. In
the language of \cite{juhl_conform} these are the residue families
$D_N^{\text{res}}(g, \lambda)$, where $g$ is
the metric \eqref{minkowskimetric}.

On the
$n$-dimensional Lorentzian upper half space with the metric
$$
x_n^{-2} \left( - dx_1^2 + \sum_{i=2}^{n} d x_i^2 \right)
$$ 
we consider formal approximate eigenfunctions of the Laplacian with
eigenvalue $\lambda(n-1-\lambda)$. The ansatz
$$
u(x) \sim \sum_{j \geq 0} x_n^{\lambda+2j} c_{2j}(x') , \,\,\,\, x = (x',x_n),
$$ 
yields recursive relations for the coefficients $c_{2j}$ so that
all coefficients are determined by the leading one $c_0$. More
precisely, there are differential operators
$$ 
T_{2j}(\lambda): C^{\infty}(\mathbb{R}^{n-1}) \rightarrow C^{\infty}(\mathbb{R}^{n-1})
$$ 
of order $2j$ (depending on $\lambda$) such that $T_{2j}(\lambda) c_{0} =
c_{2j}$. It is easily seen that
\begin{equation}\label{DEFS2}
T_{2j}(\lambda) = A_{2j}(\lambda) (\Delta_{\mathbb{M}^{n-1}})^{j},
\end{equation}
where the coefficients $A_{2j}$ satisfy the recursive relations
\begin{equation}
A_{2j-2}(\lambda) + 2j(2j+2\lambda+1-n) A_{2j}(\lambda) = 0, \; A_0(\lambda)=1.
\end{equation}
Here the d'Alembertian
$$
\Delta_{\mathbb{M}^{n-1}} = -\frac{\partial^2}{\partial x_1^2} +
\sum_{i=2}^{n-1} \frac{\partial^2}{\partial x_i^2}
$$ 
is the Laplacian of Minkowski space. Let
$$
S_{2N}(\lambda): C^{\infty}(\mathbb{R}^n) \rightarrow C^{\infty}(\mathbb{R}^{n-1})
$$
be defined by
\begin{equation}\label{DEFS}
S_{2N}(\lambda) = \sum_{j=0}^{N} \frac{1}{(2N\!-\!2j)!} T_{2j}(\lambda) i^* 
\left(\frac{\partial}{\partial x_n} \right)^{2N-2j}.
\end{equation}
The following theorem shows that both constructions yield the same
results (up to normalization).

\begin{theorem}\label{MAIN} The families $S_{2N}(\lambda+n-1-2N)$ and 
$D_{2N}^{\mathrm{nc}}(\lambda)$ coincide, up to a rational function in
  $\lambda$.
\end{theorem}
Moreover, if we define
$$
S_{2N+1}(\lambda) = \sum_{j=0}^{N} \frac{1}{(2N\!+1-\!2j)!} T_{2j}(\lambda) i^* 
\left( \frac{\partial}{\partial x_n} \right)^{2N+1-2j},
$$
we get the following analogous result:
\begin{theorem}
The families $S_{2N+1}(\lambda+n-1-(2N+1))$ and
$D_{2N+1}^{\mathrm{nc}}(\lambda)$ coincide, up to a rational function in $\lambda$.
\end{theorem}

We expect that the analogous construction of equivariant differential operators
 generalizes to pseudo-Riemannian spaces with
constant negative curvature of any signature. This would give
differential intertwining operators between spherical principal series
representations of $O(p,q)$ and $O(p,q+1)$ induced from maximal parabolic
subgroups, for all $q > p \geq 1$, analogous to the operator families
treated in this paper.

The paper is organized as follows. In Section \ref{prel} we describe the
geometric situation, the structure of the Lie groups involved and the
principal series representations induced by representations of the maximal
parabolic subgroup $P^m$. In Section \ref{algebra} we treat the invariance of the
d'Alembertian and describes algebraic constructions
which are necessary for the proof of equivariance of the differential
operators $D_N(\lambda)$. Section \ref{verma} contains the proof of Theorem
\ref{I-1} and Section \ref{ind-fam} contains the proof of Theorem
\ref{I-2}. Section \ref{formal} describe the construction of the
families $D_N^{\text{nc}}(\lambda)$ in terms of the asymptotics of
eigenfunctions of the Laplacian on anti de Sitter space.



\section{Geometric preliminaries and principal series
  representations}\label{prel}\label{principal}

The $n$-dimensional anti de Sitter space, $\text{AdS}_{n}$, is the
one-sheeted hyperboloid
\begin{equation}\label{hyp}
-t_1^2-t_2^2+x_1^2+\cdots+x_{n-1}^2=-1
\end{equation}
in the space $\mathbb{R}^{n+1}$ equipped with the pseudo-Riemannian metric
$$
ds^2 = -dt_1^2 -dt_2^2 + dx_1^2 + \cdots + dx_{n-1}^2.
$$ 
The group of isometries of $\text{AdS}_{n}$ is $O(2,n-1)$, that is,
all linear transformations of $\mathbb{R}^{n+1}$ that preserve the
hyperboloid \eqref{hyp}. We let $G^n = O(2,n-1)$. Furthermore, we
assume that $n \geq 4$.

The Lie algebra $\mathfrak{g}_n=\mathfrak{g}=\mathfrak{o}(2,n-1)$
consists of all traceless matrices of the form
$$
\begin{pmatrix} A & B \\ B^{t} & D \end{pmatrix},
$$ 
where the $2 \times 2$-matrix $A$ and the $(n-1) \times (n-1)$-matrix $D$
are skew-symmetric. For $n \geq 4$ we will use the inclusion 
\begin{equation*}
i: \mathfrak{o}(2,n-1) \rightarrow \mathfrak{o}(2,n)
\end{equation*}
given by
\begin{equation}\label{inkl}
M \mapsto \begin{pmatrix} M & 0 \\ 0 & 0 \end{pmatrix}.
\end{equation}
On the Lie group level we use the corresponding inclusion
$$
i: O(2,n-1) \rightarrow O(2,n)
$$
given by
\begin{equation}\label{inkl2}
M \mapsto \begin{pmatrix} M & 0 \\ 0 & 1 \end{pmatrix}.
\end{equation}

The involution $\theta(A) = -A^{t}$ on $\mathfrak{o}(2,n-1)$ yields a
decomposition $\mathfrak{g} = \mathfrak{k} \oplus \FRAKB$, where
$\mathfrak{k}$ and $\FRAKB$ are the eigenspaces for the eigenvalue $1$
and $-1$, respectively. We choose a maximal abelian subspace of
$\FRAKB$ by
$$
\mathfrak{a}_{\FRAKB} = \text{span}_{\mathbb{R}}(H,H_2),
$$ 
where
$$
H = \begin{pmatrix} 0 & 0 & 0 & 1 & \textbf{0} \\
0 & 0 & 0 & 0 & \textbf{0} \\
0 & 0 & 0 & 0 & \textbf{0} \\
1 & 0 & 0 & 0 & \textbf{0} \\
\textbf{0} & \textbf{0} & \textbf{0} & \textbf{0} & \textbf{0} 
\end{pmatrix}, \,\,H_2 = \begin{pmatrix} 0 & 0 & 0 & 0 & \textbf{0} \\
0 & 0 & 1 & 0 & \textbf{0} \\
0 & 1 & 0 & 0 & \textbf{0} \\
0 & 0 & 0 & 0 & \textbf{0} \\
\textbf{0} & \textbf{0} & \textbf{0} & \textbf{0} & \textbf{0}
\end{pmatrix}.
$$ 
Here and henceforth, we write elements in $\mathfrak{o}(2,n-1)$ as matrices of the block forms
\begin{equation*}
\begin{pmatrix}
4 \times 4 & 4 \times (n-3) \\
(n-3) \times 4 & (n-3) \times (n-3)
\end{pmatrix}.
\end{equation*}

We define $f_i \in (\mathfrak{a}_\FRAKB)^{*}$ by $f_i(H_j) =
\delta_{ij}$. We choose $\Sigma^+ = \{f_1, f_2, f_1+f_2, f_1-f_2 \}$
as the set of positive restricted roots. The simple restricted roots
are $f_1, f_2$ and $f_1-f_2$. The set of restricted roots is $\Sigma =
\Sigma^+ \cup - \Sigma^+$, and we have the root decomposition
$$ 
\mathfrak{g} = \mathfrak{g}_{0} \oplus \sum_{\alpha \in \Sigma}
\mathfrak{g}_{\alpha} \;, \;\mathfrak{g}_{0} = \mathfrak{a}_{\FRAKB}
\oplus \mathfrak{m}_{\FRAKB},
$$ 
where $\mathfrak{m}_{\FRAKB}$ is the centralizer of
$\mathfrak{a}_{\FRAKB}$ in $\mathfrak{k}$. We have
$$
\mathfrak{m}_{\FRAKB} = \text{span}_{\mathbb{R}}( M_{ij} \,\vert \, 1 \leq i, 
j \leq n-3 ),
$$ 
where for $1 \leq i,j \leq n-3$ we set
$$ 
M_{ij} = \begin{pmatrix} 0 & 0 & 0 & 0 & \textbf{0} \\
0 & 0 & 0 & 0 & \textbf{0} \\
0 & 0 & 0 & 0 & \textbf{0} \\
0 & 0 & 0 & 0 & \textbf{0} \\
\textbf{0} & \textbf{0} & \textbf{0} & \textbf{0} & R_{ij}
\end{pmatrix}
\,\, \text{with} \,\, (R_{ij})_{rs} = \delta_{ir}\delta_{js} -
\delta_{is}\delta_{jr}.
$$ 
We choose a basis of root vectors as follows. For $1 \leq j \leq n-3$, $e_j \in
\mathbb{R}^{n-3}$, let
$$
Y^{+}_{j} = \begin{pmatrix}
0 & 0 & 0 & 0 & e_j \\
0 & 0 & 0 & 0 & \textbf{0} \\
0 & 0 & 0 & 0 & \textbf{0} \\
0 & 0 & 0 & 0 & e_j \\
e_j^t & \textbf{0} & \textbf{0} & -e_j^{t} 
&  \textbf{0} \\
\end{pmatrix}  \in  \mathfrak{g}_{f_1},\,\, Z^{+}_{j} = \begin{pmatrix}
0 & 0 & 0 & 0 & \textbf{0} \\
0 & 0 & 0 & 0 & e_j \\
0 & 0 & 0 & 0 & e_j \\
0 & 0 & 0 & 0 & \textbf{0} \\
\textbf{0} & e_j^t & -e_j^{t} & \textbf{0} 
& \textbf{0} \\
\end{pmatrix} \in \mathfrak{g}_{f_2},
$$
and  let $Y_{j}^- = (Y_j^+)^{t} \in \mathfrak{g}_{-f_1}$, $Z_{j}^{-} =
(Z_j^+)^{t} \in \mathfrak{g}_{-f_2}$. We define $W_1$ and $W_2$ by
$$
W_1 =  \begin{pmatrix}
0 & 1 & -1 & 0 & \textbf{0} \\
-1 & 0 & 0 & 1 & \textbf{0} \\
-1 & 0 & 0 & 1 & \textbf{0} \\
0 & 1 & -1 & 0 & \textbf{0} \\
\textbf{0} & \textbf{0} & \textbf{0} 
& \textbf{0}  &  \textbf{0} \\
\end{pmatrix} \in \mathfrak{g}_{f_1 + f_2}, W_2 = \begin{pmatrix}
0 & 1 & 1 & 0 & \textbf{0} \\
-1 & 0 & 0 & 1 & \textbf{0} \\
1 & 0 & 0 & -1 & \textbf{0} \\
0 & 1 & 1 & 0 & \textbf{0} \\
\textbf{0} & \textbf{0} & \textbf{0} 
& \textbf{0} & \textbf{0} \\
\end{pmatrix} \in \mathfrak{g}_{f_1 - f_2}.
$$
We have
$[Y_j^{+}, Z_{k}^{+}] = \delta_{jk} W_1$
and $[Y_j^{+}, Z_{k}^{-}] = \delta_{jk} W_2$. Moreover, we set
\begin{equation*}
\begin{split}
[Y_j^{-}, Z_{k}^{+}] &= \delta_{jk} W_3 = \delta_{jk} \theta(W_2) \in
\mathfrak{g}_{-f_1 + f_2}, \\
[Y_j^{-}, Z_{k}^{-}] &= \delta_{jk} W_4 
= \delta_{jk} \theta(W_1) \in \mathfrak{g}_{-f_1 - f_2}.
\end{split}
\end{equation*}
It will be convenient to introduce the following elements.
\begin{align}
Q_1^{+} &= \frac{W_1+W_2}{2}, & Q_1^{-} & = (Q_1^{+})^{t} = \frac{-W_3-W_4}{2}, \\
Q_2^{+} &= \frac{W_2-W_1}{2}, & Q_2^{-} & = (Q_2^{+})^{t} = \frac{W_4-W_3}{2}. 
\end{align}
We define $\mathfrak{n}_{\FRAKB}^{+} = \sum_{\lambda \in \Sigma^{+} }
\mathfrak{g}_{\lambda} $. A minimal parabolic subalgebra is given by $
\mathfrak{p}_{\text{min}} = \mathfrak{m}_{\FRAKB} \oplus
\mathfrak{a}_{\FRAKB} \oplus \mathfrak{n}_{\FRAKB}^{+}$, a maximal
parabolic subalgebra is given by
$\mathfrak{p}=\mathfrak{p}_{\text{min}} \oplus \mathfrak{g}_{-f_2}$,
with Langlands decomposition $ \mathfrak{p} = \mathfrak{m} \oplus
\mathfrak{a} \oplus \mathfrak{n}^{+}$, where
$$ 
\mathfrak{m} = \mathfrak{m}_{\FRAKB} \oplus \mathfrak{g}_{f_2}
\oplus \mathfrak{g}_{-f_2} \oplus \text{span}_{\mathbb{R}}(H_2),\;\;
\mathfrak{a} = \text{span}_{\mathbb{R}}( H ), \;\; \mathfrak{n}^{+}
= \mathfrak{g}_{f_1} \oplus \mathfrak{g}_{f_1+f_2} \oplus
\mathfrak{g}_{f_1-f_2}.
$$
We have $\mathfrak{n}^{-} = \theta(\mathfrak{n}^+) =
\mathfrak{g}_{-f_1} \oplus \mathfrak{g}_{-f_1-f_2} \oplus
\mathfrak{g}_{-f_1+f_2}$. We have the decomposition $ \mathfrak{g} =
\mathfrak{n}^- \oplus \mathfrak{m} \oplus \mathfrak{a} \oplus
\mathfrak{n}^{+}$. The commutator relations for $\mathfrak{o}(2,n-1)$
can be found in the appendix. We have
\begin{equation}\label{bracket_rel}
\begin{split}
[ \mathfrak{m}, \mathfrak{n}^+ ] \subset \mathfrak{n}^+, [
\mathfrak{m}, \mathfrak{n}^- ] \subset \mathfrak{n}^-, [ \mathfrak{m},
\mathfrak{m} ] \subset \mathfrak{m}, [\mathfrak{m}, \mathfrak{a}] = 0,
[ \mathfrak{n}^-, \mathfrak{n}^- ] = 0.
\end{split}
\end{equation}
We will also use the following basis for $\mathfrak{n}^{-}_n$:
\begin{equation}
I_j = \begin{cases} Y_{j}^{-} & \text{for} \; 1 \leq j \leq n-3, \\ 
Q_1^- & \text{for} \; j = n-2, \\ Q_2^- & \text{for} \; j = n-1. 
\end{cases}
\label{NBASIS}
\end{equation}

Following \cite[p. 133]{knapp}, we define $K_0$, $A$, $N^+$ and $M_0$
as the analytic subgroups of $G^n = O(2,n-1)$ with Lie algebras
$\mathfrak{k}$, $\mathfrak{a}$, $\mathfrak{n}^{+}$ and $\mathfrak{m}$,
respectively. We let $K$ be the maximal compact subgroup $O(2) \times
O(n-1)$. We will also use the notation $K_0^n$, $K^n$, $(N^+)^n$ and
$M_0^n$ for these subgroups of $G^n$. The Lie algebras
$\mathfrak{n}^{+}$ and $\mathfrak{a}$ are abelian, and hence
$N^+=\exp(\mathfrak{n}^+)$, $A=\exp(\mathfrak{a})$.

We have a bijection $\mathfrak{so}(1,n-2) \simeq \mathfrak{m}$ which
shows that $M_0 \simeq SO_{0}(1,n-2)$. In \cite{nishikawa} it is
proved that, for $0 \leq p \leq q$, we have $\exp(\mathfrak{so}(p,q))
= SO_0(p,q)$ if and only if $p=0$ or $p=1$. This shows that
$\exp(\mathfrak{m}) = M_0$.

We have
$$
Z_{K}(\mathfrak{a}) = \{ k \in K \; \vert \; \Ad(k)=1 \; \text{on} \,\; 
\mathfrak{a}\, \} = \{ k \in O(2) \times O(n-1) \,\, \vert \,\, kH = Hk \}.
$$
Every $m \in Z_K(\mathfrak{a})$ satisfies
$$
m = \begin{pmatrix}
a & 0 & 0 & 0 & 0 & \dots & 0 \\
0 & b & 0 & 0 & 0 & \dots & 0 \\
0 & 0 & m_{11} & 0 & m_{13} & \dots & m_{1 (n-1)} \\
0 & 0 & 0 & a & 0 & \dots & 0 \\
0 & 0 & m_{31} & 0 & m_{33} & \dots & m_{3 (n-1)} \\
\vdots & \vdots & \vdots & \vdots & \vdots & & \vdots \\
0 & 0 & m_{(n-1) 1} & 0 & m_{(n-1) 3 } & \dots & m_{(n-1) (n-1) }
\end{pmatrix} \in O(2) \times O(n-1),
$$
where $a = \pm 1, b= \pm 1$. We know that
$\exp(Z_{\mathfrak{k}}(\mathfrak{a}))$ consists of matrices of the form
$$
\begin{pmatrix}
1 & 0 & 0 & 0 & 0 & \dots & 0 \\
0 & 1 & 0 & 0 & 0 & \dots & 0 \\
0 & 0 & m_{11} & 0 & m_{13} & \dots & m_{1 (n-1)} \\
0 & 0 & 0 & 1 & 0 & \dots & 0 \\
0 & 0 & m_{31} & 0 & m_{33} & \dots & m_{3 (n-1)} \\
\vdots & \vdots & \vdots & \vdots & \vdots & & \vdots \\
0 & 0 & m_{(n-1) 1} & 0 & m_{(n-1) 3 } & \dots & m_{(n-1) (n-1) }
\end{pmatrix} \in SO(2) \times SO(n-1),
$$ 
and we conclude that 
$$
Z_K(\mathfrak{a}) = \{ 1_{(n)}, J_{(n)} \} \langle w_1,
w_2 \rangle \exp(Z_{\mathfrak{k}}(\mathfrak{a})),
$$ 
where
\begin{equation}
\begin{split}
\label{w1234}
w_1 = \begin{pmatrix} 1 & 0 & 0 & 0 & \textbf{0} \\
0 & -1 & 0 & 0 & \textbf{0} \\
0 & 0 & 1 & 0 & \textbf{0} \\
0 & 0 & 0 & 1 & \textbf{0} \\
\textbf{0} & \textbf{0} &  \textbf{0} & \textbf{0} & \textbf{1} 
\end{pmatrix}, \,\,w_2 = \begin{pmatrix} 1 & 0 & 0 & 0 & \textbf{0} \\
0 & 1 & 0 & 0 & \textbf{0} \\
0 & 0 & -1 & 0 & \textbf{0} \\
0 & 0 & 0 & 1 & \textbf{0} \\
\textbf{0} & \textbf{0} &  \textbf{0} & \textbf{0} & \textbf{1} 
\end{pmatrix},
\end{split}
\end{equation}
\begin{equation}
J_{(n)} = \begin{pmatrix} -1 & 0 & 0 & 0 & \textbf{0} \\
0 & -1 & 0 & 0 & \textbf{0} \\
0 & 0 & -1 & 0 & \textbf{0} \\
0 & 0 & 0 & -1 & \textbf{0} \\
\textbf{0} & \textbf{0} &  \textbf{0} & \textbf{0} & -\textbf{1} 
\end{pmatrix},
\end{equation}
and $\langle w_1, w_2 \rangle$ is the subgroup of $O(2,n-1)$ generated
by $w_1$ and $w_2$. Here $1_{(n)}$ is the identity element in $G^n$.

We observe that $Z_{\mathfrak{k}}(\mathfrak{a}) \subset \mathfrak{m}$
and $\exp( Z_{\mathfrak{k}}(\mathfrak{a}) ) \subset M_0$. By
definition, $M=Z_{K}(\mathfrak{a}) M_0$ and hence
\begin{multline}
M=Z_K(\mathfrak{a}) M_0 = \bigg( \{ 1_{(n)}, J_{(n)} \} \langle w_1,
w_2 \rangle \exp(Z_{\mathfrak{k}}(\mathfrak{a})) \bigg) M_0 \\ = \{
1_{(n)}, J_{(n)} \} \langle w_1, w_2 \rangle M_0 \subset G^n =
O(2,n-1). \label{MEGENSKAP}
\end{multline}
We have $M_0 \simeq SO_{0}(1,n-2)$ and hence $\langle w_1, w_2 \rangle
M_0 \simeq O(1,n-2)$. The maximal parabolic subgroup is given by
$P=MAN^+$.

Furthermore, we see that $M_0 \cap K \simeq SO(n-2)$, the elements of
$M_0 \cap K$ are given by
$$
\begin{pmatrix} 1 & 0 & 0      & 0   & 0      & \dots & 0 \\ 
                0 & 1 & 0    & 0   & 0    & \dots & 0 \\
                0 & 0 & d_{1 1} & 0 & d_{1 2} & \dots & d_{1 (n-2)} \\
                0 & 0   & 0       & 1 & 0       & \dots & 0           \\
                0 & 0 & d_{2 1} & 0 & d_{2 2} & \dots & d_{2 (n-2)} \\
                \vdots & \vdots & \vdots & \vdots &\vdots &  & \vdots \\
                0 & 0 & d_{(n-2) 1} & 0 & d_{(n-2) 2} & \dots & d_{(n-2) (n-2)}
\end{pmatrix} \in SO(2,n-1),
$$
and the elements of $\langle w_1, w_2 \rangle M_0 \cap K \simeq \{ \pm 1 \}
\times O(n-2)$ are given by
$$
 \begin{pmatrix}1 & 0 & 0      & 0   & 0      & \dots & 0 \\ 
                0 & \pm 1 & 0    & 0   & 0    & \dots & 0 \\
                0 & 0 & d_{1 1} & 0 & d_{1 2} & \dots & d_{1 (n-2)} \\
                0 & 0   & 0       & 1 & 0       & \dots & 0           \\
                0 & 0 & d_{2 1} & 0& d_{2 2} & \dots & d_{2 (n-2)} \\
                \vdots & \vdots & \vdots & \vdots &\vdots &  & \vdots \\
                0 & 0 &d_{(n-2) 1} &0& d_{(n-2) 2} & \dots & d_{(n-2) (n-2)}
\end{pmatrix} \in O(2,n-1).
$$

\begin{lemma} $K/(\langle w_1, w_2 \rangle M_0 \cap K) = S^1 \times S^{n-2}$, 
and $$K/(M \cap K) \simeq (S^1 \times S^{n-2})/ \mathbb{Z}_2,$$where
the $\mathbb{Z}_2$-action is defined by
\begin{equation}
(x,y) \mapsto (-x,-y) \,\, \text{on} \,\, S^1 \times S^{n-2}.
\label{involution}
\end{equation}
\end{lemma}

\begin{proof} We have $\langle w_1, w_2 \rangle M_0 \cap K \simeq O(1) 
\times O(n-2)$ and hence 
\begin{equation}
K/(\langle w_1, w_2 \rangle M_0 \cap K)
\simeq (O(2) \times O(n-1)) / (O(1) \times O(n-2)) \simeq S^1 \times S^{n-2}.
\label{KFORMEL}
\end{equation}
Using $M \cap K = \{ 1_{(n)}, J_{(n)} \} (\langle w_1, w_2 \rangle M_0
\cap K)$, we conclude that 
\begin{equation}
K/(M \cap K) \simeq (S^1 \times S^{n-2}) /
\mathbb{Z}_2.
\label{KFORMEL2}
\end{equation}
\end{proof}

Let $C_n$ be given by \eqref{lightcone}. We define $\text{Ein}_{n-1}$, the Einstein universe, as the set of
all straight lines in $C_n$ through the origin; see
\cite{frances_conformal} and \cite{falk}. We define $\widehat{\text{Ein}}_{n-1}$ as
the set of all rays in $C_n$ from the origin. The sphere with radius
$\sqrt{2}$ in $\mathbb{R}^{n+1}$ with center at the origin intersects
$C_n$ in the set $S^1 \times S^{n-2}$, and hence
$\widehat{\text{Ein}}_{n-1} \simeq S^1 \times S^{n-2}$ and
$\text{Ein}_{n-1} \simeq (S^1 \times S^{n-2})/\mathbb{Z}_2$, where the
where the $\mathbb{Z}_2$-action is coming from the involution
\eqref{involution}.
\begin{lemma}
$G^n=O(2,n-1)$ acts transitively on $\widehat{\text{Ein}}_{n-1}$, and
  $$\hat{P} = \langle w_1, w_2
  \rangle M_0 AN^+$$ is the isotropy group of the ray generated by
  $(1,0,0,1,0,\dots,0) \in \mathbb{R}^{n+1}$. It follows that $$
  \widehat{\mathrm{Ein}}_{n-1} \simeq G^n / \PHAT^n.$$
\end{lemma}
\begin{proof} Let $\tilde{P}^n$ denote the isotropy group of the ray 
generated by $$(1,0,0,1,0,\dots,0) \in \mathbb{R}^{n+1}.$$It is
immediate that $\langle w_1, w_2 \rangle M_0 AN^+ \subset
\tilde{P}^n$. We now prove $\tilde{P}^n \subset \hat{P}^n$. Let $g \in
\tilde{P}^n$. First, assume that $g \in \tilde{P}^n \cap
SO_{0}(2,n-1)$ and write $g=kan \in K_0
\exp(\mathfrak{a}_{\mathfrak{q}})
\exp(\mathfrak{n}^+_{\mathfrak{q}})$. We have
$\exp(\mathfrak{a}_{\mathfrak{q}}) \subset \tilde{P}^n$,
$\exp(\mathfrak{n}^+_{\mathfrak{q}}) \subset \tilde{P}^n$ and hence $k
\in \tilde{P}^n$. It is easy to see that $K_0 \cap \tilde{P}^n \subset
M_0$ and hence we conclude $k \in M_0$. Because
$\exp(\mathfrak{n}_{\mathfrak{q}}^+) \subset N^+$ and
$\exp(\mathfrak{a}_{\mathfrak{q}}) \subset M_0 A$ we conclude that $g
\in M_0 AN^+$. Finally, using $O(2,n-1)=\langle w_1, w_2 \rangle
SO_{0}(2,n-1)$ we conclude that $\tilde{P}^n \subset \langle w_1, w_2
\rangle M_0 A N^+$.
\end{proof}

\begin{lemma}
$G^n=O(2,n-1)$ acts
transitively on $\EIN_{n-1}$, and $P^n$ is the isotropy group of the
line generated by $(1,0,0,1,0,\dots,0) \in \mathbb{R}^{n+1}$. It follows
that $$\mathrm{Ein}_{n-1} \simeq G^n/P^n.$$

\end{lemma}
\begin{proof} 
Let $\tilde{P}^n$ denote the isotropy group of the line generated by
$$(1,0,0,1,0,\dots,0) \in \mathbb{R}^{n+1}.$$It is immediate that $P^n
\subset \tilde{P}^n$. We now prove $\tilde{P}^n \subset P^n$. Let $g
\in \tilde{P}^n$. First, assume that $g \in \tilde{P}^n \cap
SO_{0}(2,n-1)$, and write $g=kan \in K_0
\exp(\mathfrak{a}_{\mathfrak{q}})
\exp(\mathfrak{n}^+_{\mathfrak{q}})$. We have
$\exp(\mathfrak{a}_{\mathfrak{q}}) \subset \tilde{P}^n$,
$\exp(\mathfrak{n}^+_{\mathfrak{q}}) \subset \tilde{P}^n$ and hence $k
\in \tilde{P}^n$. It is easy to see that $K_0 \cap \tilde{P}^n \subset
M$ and hence we conclude $k \in M$. Because
$\exp(\mathfrak{n}_{\mathfrak{q}}^+) \subset N^+$ and
$\exp(\mathfrak{a}_{\mathfrak{q}}) \subset MA$ we conclude that $g \in
P^n$. Finally, using $O(2,n-1) = \langle w_1, w_2 \rangle
SO_{0}(2,n-1)$, $\langle w_1, w_2 \rangle \subset P^n$, we conclude
that $\tilde{P}^n \subset P^n$.
\end{proof}

We will use the notation
$$
 n^{-}_{(x_1, x_2, x_3, \dots, x_{n-1})} = \exp(x_1 Q_1^- + x_2 Q_2^- + \sum_{i=3}^{n-1} x_{i}
Y_{i-2}^- ),
$$
and $(N^-)^n$ is identified with $\mathbb{R}^{n-1} = \mathbb{M}^{n-1}$ by
\begin{equation}\label{IDENT}
\mathbb{M}^{n-1} \ni (x_1, x_2, x_3, \dots,
x_{n-1}) \mapsto n^{-}_{(x_1, x_2, x_3, \dots, x_{n-1})} \in (N^-)^n.
\end{equation}
We let $\pi: C_n \setminus \{ 0\} \rightarrow \EIN_{n-1}$ denote the
natural projection and define the map 
\begin{equation}
p: G^{n} \rightarrow
\EIN_{n-1}, \,x \mapsto \pi \left( x(1,0,0,1,0_{n-3}) \right).
\end{equation}
Here $(1,0,0,1,0_{n-3})$ denotes the vector $(1,0,0,1,0,\dots,0) \in
\mathbb{R}^{n+1}$, and $$x(1,0,0,1,0_{n-3})$$ denotes matrix
multiplication of $x \in (N^{-})^n \subset G^{n}$ with the column matrix
$(1,0,0,1,0_{n-3})^t$.
\begin{lemma}
\label{GE_L}
The injection $$j: \mathbb{M}^{n-1} \simeq (N^-)^n \rightarrow
\EIN_{n-1}, \,x \mapsto p(x)$$ has image $\EIN_{n-1}
\setminus \pi(D_n)$, where
$$
D_n = \{ (t_1, t_2, x_1, x_2, x_3, \dots, x_{n-1}) \, \vert \, t_1 + x_2 =
0\}.
$$

The injection $j$ gives a conformal compactification of
$\mathbb{M}^{n-1}$ in $\EIN_{n-1}$. 
\end{lemma}

\begin{proof} We compute
\begin{equation*}
\begin{split}
& \exp( w_1 Q_1^- + w_2 Q_2^- + \sum_{j=1}^{n-3} w_{j-2} Y_{j}^{-})
(1,0,0,1,0_{n-3}) = \\ & = (1 - w_1^2 + w_2^2+ \vert w' \vert^2, 2w_1,
2w_2, 1 + w_1^2 - w_2^2 - \vert w' \vert^2, 2w_3, 2w_4, \ldots,
2w_{n-1}),
\end{split}
\end{equation*}
where we let $w'=(w_3, \dots, w_{n-1})$. The equations
\begin{equation*}
\begin{split}
t_1 = r(1-w_1^2+w_2^2+\vert w' \vert^2), \\
t_2 = 2rw_1, \\ 
x_1 = 2rw_2, \\
x_2 = r(1+w_1^2-w_2^2-\vert w' \vert^2), \\
x_3 = 2rw_3, \\
x_4 = 2rw_4, \\
\dots \\
x_{n-1} = 2rw_{n-1},
\end{split}
\end{equation*}
imply that $r=\frac{t_1+x_2}{2}$ and hence
\begin{equation}\label{fre}
w_1 = \frac{t_2}{t_1+x_2}, \\
w_2 = \frac{x_1}{t_1+x_2}, \\
w_3 = \frac{x_3}{t_1+x_2}, \\
w_4 = \frac{x_4}{t_1+x_2}, \\
\dots, \\
w_{n-1} = \frac{x_{n-1}}{t_1+x_2}.
\end{equation}
\end{proof}
If we let $H^n \subset G^n$ denote the elements $g \in G^n$ such that $p(g)
\in \pi(D_n)$ we get the Bruhat decomposition
$$
G^{n} = (N^-)^{n} M^{n} A (N^+)^{n} \cup H^{n}.
$$
For $g \in G^n, g \not \in H^n$ we write $g=\eta(g) m(g) a(g) n$. 

\begin{lemma}\label{extlemma} Let $\xi$ be a character of $A$. Then every 
$f \in C_0^{\infty}((N^-)^n)$ extends to a function $\hat{f}
\in C^{\infty}(G^n)$ such that $\hat{f}(x\tilde{m}\tilde{a}\tilde{n})
= \xi(\tilde{a}) \hat{f}(x)$ for all $\tilde{m} \tilde{a} \tilde{n}
\in P^n$ and $x \in G^n$.
\end{lemma}

\begin{proof}

We define $\hat{f}(n^- \tilde{m} \tilde{a} \tilde{n} ) =
f(n^-) \xi(\tilde{a})$, where $n^- \in N^-, \tilde{m} \in M, \tilde{a}
\in A$ and $\tilde{n} \in N^+$, and $\hat{f}(g)=0$ for $g \in H^n$. To prove that $\hat{f}$ is smooth, it is sufficient to prove that for every
$h \in H^n$ there exists an open neighborhood of $h$ (in $G^n$), where
$\hat{f}$ is identically zero.

Assume that $h \in H^n$. Then $p(h) \in \pi(D_n)$. For any $\epsilon > 0$
the set $$
U(\epsilon) := \{ \pi( (t_1, t_2, x_1, x_2, x_3, \dots, x_{n-1}) )\, \vert \, t_1 + x_2 <
\epsilon \}
$$
is a neighborhood of
$p(h)$, and using the continuity of $p$ the set $p^{-1}( U(\epsilon) )$ is
a neighborhood of $h$. Now let $g \in G^n$. If $g P^n \in \text{Ein}_{n-1}$ is given by $\pi( (t_1, t_2, x_1, \dots,
x_{n-1}) )$ then using \eqref{fre} we get
$$\eta(g) =
n^{-}_{ \left(\frac{t_2}{t_1+x_2},\frac{x_1}{t_1+x_2},\frac{x_3}{t_1+x_2},
  \dots, \frac{x_{n-1}}{t_1+x_2} \right)}  ,$$

and using the compact support of $f$, we see that for sufficiently small
$\epsilon > 0$ the function $\hat{f}$ is identically zero in $p^{-1}(U(\epsilon))$.
\end{proof}
Let $R$ be the right-regular representation of $G$ on $C^\infty(G)$,
i.e., $(R(g)u)(x) = u(xg)$. We let $\mathcal{U}(\mathfrak{g})$
operate from the \emph{right} on $C^{\infty}(G)$ through the extension
of its differential
\begin{equation}\label{def-I}
(R(X)u)(x) = \frac{d}{dt} \Big\vert_{t=0} u(x \exp(tX)), \;\; X \in
\mathfrak{g}.
\end{equation}
Observe that for any $u \in C^{\infty}(G)$ and $X \in \mathfrak{g}$ we have
\begin{equation}\label{ad_komm}
R(\Ad(g)X ) u = R(g) \circ R(X) \circ R({g^{-1}}) u.
\end{equation}
Using $R$, elements of $\mathcal{U}(\mathfrak{g})$ induce
left-invariant differential operators on $C^{\infty}(G)$.

We now define principal series representations of
$G^m=O(2,m-1)$ induced from the parabolic subgroups $P^m, P_0^m$ and
$\hat{P}^m$ with respective Langlands decompositions
$$
P=MAN^+, \; P_0=M_0AN^+, \; \hat{P}=\hat{M} AN^+,
$$
where $\hat{M}=\langle w_1, w_2 \rangle M_0$.

We first define principal series as induced representations. For
$\lambda \in \c$ with corresponding character $\xi_{\lambda}$ and a
representation $\sigma: M \rightarrow \mathbb{C}^*$ let
\begin{multline*}
\text{Ind}_{P^m}^{G^m}(\xi_\lambda \otimes \sigma) = \\ \{ u \in
C^{\infty}(G^m) \, \vert \, u(g\tilde{m}\tilde{a} \tilde{n}) =
\sigma(\tilde{m}) \xi_{\lambda} (\tilde{a}) u(g),
\tilde{m}\Tilde{a}\tilde{n} \in P^m, g \in G^m \}.
\end{multline*}
$G^m$ acts on $\text{Ind}_{P^m}^{G^m}(\xi_\lambda \otimes \sigma)$ by
left translation, i.e.,
$$
(\pi_{\xi_{\lambda} \otimes \sigma}(g) u)(x) = u(g^{-1}x).
$$
Note that our definition differs from the convention in \cite{knapp}. 

We use the notation $\text{Ind}_{P^m}^{G^m}(\xi_\lambda) =
\text{Ind}_{P^m}^{G^m}(\xi_\lambda \otimes \sigma_+)$, where
$\sigma_+$ is the trivial representation of $M$. We also use the
notation $\pi_{\xi_{\lambda}}(g) = \pi_{\xi_{\lambda} \otimes
\sigma_+}(g)$. Let the representation $\sigma_{-}: M^{m} \rightarrow
\mathbb{C}$ be the unique representation such that $\sigma_{-}(x)=1$
for $x \in \langle w_1, w_2 \rangle M_0$, $\sigma_{-}(J_{(m)}) = -1$.

The principal series representations of $G^m$ induced from $\hat{P}^m$
and $P_0^m$ are defined in a similar way.

Let $K_M: = K \cap M$. In order to describe the \emph{compact
realization} of the principal series representation of $G^m$ (induced
from $P^m$), let
$$ 
C^\infty(K)^{K_M} = \{ F \in C^\infty(K) \,|\; F(k\tilde{m}) =
F(k), \; \tilde{m} \in K_M \}.
$$ 
Let $\delta: \text{Ind}_{P}^{G}(\xi_\lambda) \rightarrow
C^{\infty}(K)^{K_M}$ be the restriction map. Using the decomposition
$G = K M A N^+$ (\cite[prop 7.83(g)]{knapp2}), we write $g = \kappa(g)
\mu(g) e^{H(g)} \eta(g)$ and find
$$
\delta^{-1}(F)(g) = F(\kappa(g)) \xi_{\lambda}(e^{H(g)}).
$$ 
$\delta$ becomes an isomorphism of $G$-modules if on
$C^{\infty}(K)^{K_M}$ we define  
\begin{equation}
(\pi_{\lambda}^{\text{c}}(g) F)(k K_M) = F(\kappa(g^{-1}k) K_M)
\xi_{\lambda}(e^{H(g^{-1}k)}).
\label{EINMOD}
\end{equation}
Now \eqref{KFORMEL} gives a bijection
\begin{equation}\label{bij-1}
C^{\infty}(K^m)^{K^m_M} \rightarrow C^{\infty}(\EIN_{m-1}),
\end{equation}
and the equation \eqref{EINMOD} defines a $G$-module structure on
$C^{\infty}(\text{Ein}_{m-1})$. If we want to emphasize the parameter
$\lambda \in \mathbb{C}$ of $\pi_{\lambda}^{\text{c}}$ we write
$C^{\infty}(\EIN_{m-1})_{\lambda}$ for the representation space.

Let $K_{\hat{M}} = K \cap \hat{M}$. In order to describe the
\emph{compact realization} of the principal series representation of
$G^m$ (induced from $\PHAT^m$) let
$$ 
C^{\infty}(K)^{K_{\hat{M}}} = \{ F \in C^{\infty}(K)
\; | \; F(k \hat{m}) = F(k), \; \hat{m} \in K_{\hat{M}} \}.
$$ 
Let $\hat{\delta}: \text{Ind}_{\hat{P}}^{G}(\xi_\lambda) \rightarrow
C^{\infty}(K)^{K_{\hat{M}}}$ be the restriction map. Using the
decomposition $G = K \hat{M} A N^+$, we write $g = \kappa(g) \mu(g)
e^{H(g)} \eta(g)$ and find
$$
\hat{\delta}^{-1}(F)(g) = F(\kappa(g)) \xi_{\lambda}(e^{H(g)}).
$$
$\hat{\delta}$ becomes an isomorphism of $G$-modules if on
$C^{\infty}(K)^{K_{\hat{M}}}$ we define
$$ 
(\pi_{\lambda}^{\text{c}}(g) F)(kK_{\hat{M}}) =
F(\kappa(g^{-1}k)K_{\hat{M}}) \xi_{\lambda}(e^{H(g^{-1}k)}).
$$
Now \eqref{KFORMEL2} gives a bijection
\begin{equation}\label{bij-2}
C^{\infty}(K^m)^{K_{\hat{M}}^m} \rightarrow
C^{\infty}(\widehat{\EIN}_{m-1}).
\end{equation}
If we want to emphasize the parameter $\lambda \in \mathbb{C}$ of $\pi_{\lambda}^{\text{c}}$, we write
$C^{\infty}(\widehat{\EIN}_{m-1})_{\lambda}$ for the representation space.

We describe the \emph{non-compact realization} of the principal series
representation (of $G^m$ induced from $P^m$). Let $\beta_{\xi_{\lambda}
  \otimes \sigma}:
\text{Ind}_{P^m}^{G^m}(\xi_\lambda \otimes \sigma) \rightarrow C^{\infty}((N^-)^m)
\simeq C^{\infty}(\mathbb{R}^{m-1})$ be the restriction map. We define
$$
C^{\infty}((N^{-})^m )_{\xi_\lambda \otimes \sigma} := \beta_{\xi_{\lambda}
  \otimes \sigma}\left(
\text{Ind}_{P^m}^{G^m}(\xi_\lambda \otimes \sigma) \right).
$$
We know that $G^m=(N^-)^m M^m A (N^+)^m \cup H^m$, where $H^m \subset G^m$ is a
lower-dimensional set, and write $g= \eta(g) m(g) a(g) n$. We find
$$ 
(\beta_{\xi_{\lambda} \otimes \sigma} )^{-1}(F)(n^- man) = F(n^-)
\sigma(m) \xi_{\lambda}(a).
$$
The definition
$$ 
(\pi_{\xi_{\lambda} \otimes \sigma}^{\text{nc}}(g)F)(x) =
F(\eta(g^{-1}x)) \sigma(m(g^{-1}x)) \xi_{\lambda}(a(g^{-1} x)), \; F
\in C^{\infty}((N^-)^m)_{\xi_{\lambda} \otimes \sigma}
$$ 
for $x \in (N^-)^m$ makes $\beta_{\xi_{\lambda} \otimes \sigma}: 
\text{Ind}_{P^m}^{G^m}(\xi_\lambda) \rightarrow
C^{\infty}((N^-)^m)_{\xi_{\lambda} \otimes \sigma} $ into a
$G^m$-equivariant map.

\begin{lemma} \label{aux_lemma}
Let $u \in \mathcal{U}(\mathfrak{n}_{n+1}^{-})$ such that for every
$f\in \mathrm{Ind}^{G^{n+1}}_{P^{n+1}}(\xi_{\lambda})$ we have
$$
(u f)(e) = 0,
$$
where $e$ is the identity element of $G^{n+1}$. This implies that $u = 0$.
\end{lemma}

\begin{proof} We use the notation \eqref{NBASIS}. Observe that that 
$N^{-}_{n+1} = \exp(\mathfrak{n}_{n+1}^{-})$. Let $f
\left(\exp(\sum_{i=1}^{n} x_i I_{i}) \right) = \xi(x_1, \dots, x_n)
p(x_1, \dots, x_n)$, where $p$ is a polynomial and $\xi$ a smooth
cut-off function such that $\xi(x_1, \dots, x_n) = 1$ for $x_1^2 +
\dots + x_n^2 < 1$ and $\xi(x_1, \dots, x_n) = 0$ for $x_1^2 + \dots +
x_n^2 > 2$. According to Lemma \ref{extlemma} $f$ extends to a smooth
function $f \in \text{Ind}^{G^{n+1}}_{P^{n+1}}(\xi_{\lambda})$. Note
that for \emph{any} polynomial $p$ and $u \in
\mathcal{U}(\mathfrak{n}_{n+1}^{-})$, $u=(I_{1})^{k_1} \dots
(I_{n})^{k_n}$ we get
$$
(uf)(e) = \left( \left(
\frac{\partial}{\partial x_1} \right)^{k_1} \left(
\frac{\partial}{\partial x_2} \right)^{k_2} \dots \left(
\frac{\partial}{\partial x_n}\right)^{k_n} \right) p (0).
$$ 
We see that if we let $p(x_1, \dots, x_n) = x_1^{l_1} \dots
x_n^{l_n}$ then $(uf)(e) = 0$ if and only if $u \in
\mathcal{U}(\mathfrak{n}^{-}_{n+1})$ contains no monomial of the form
$I_{1}^{l_1} \dots I_{n}^{l_n}$. This completes the proof.
\end{proof}

\section{Powers of the Laplacian and algebraic constructions}
\label{algebra}
\label{invariants}

Under the identification \eqref{IDENT}, the d'Alembertian
$\Delta_{\mathbb{M}^{n-1}}$ is induced by
\begin{equation}
\LAP_{n-1} = -(Q_1^-)^2 + (Q_2^-)^2 + \sum_{j=1}^{n-3} (Y_{j}^{-})^2
\in \mathcal{U}(\mathfrak{n}_n^-) \subset \mathcal{U}(\mathfrak{g}_n).
\label{LAP}
\end{equation}
$\LAP_{n-1}$ commutes with $\mathfrak{m}_n$. Moreover,

\begin{proposition}\label{kommprop} $P \in 
\mathcal{U}(\mathfrak{n}_{n}^{-})$ satisfies $[X,P]=0$ for all $X
\in \mathfrak{m}_n$ if and only if $P$ is a polynomial in $\LAP_{n-1}$.
\end{proposition}

\begin{proof} Although the result is well-known, we give a proof for the 
sake of completeness; the argument is similar to
\cite[p. 270-271]{helgason_diff}.  It is easy to see that if $P$ is a
polynomial in $\LAP_{n-1}$, then $[X,P]=0$ for all $X \in
\mathfrak{m}_n$. We now prove that if $[X,P]=0$ for all $X \in
\mathfrak{m}_n$, then $P$ is a polynomial in $\LAP_{n-1}$. Observe that $M_0^n = \exp(\mathfrak{m}_n)$ is generated
by $A_j(t), B_j(t), C_j(t), D_{i j}(t)$, $1 \leq i, j \leq n-3$, $t
\in \mathbb{R}$, where
\begin{eqnarray*}
& A_j(t) = \exp \left( \frac{Z_j^+ + Z_j^-}{2} t \right), & B_j(t) =
\exp \left( \frac{Z_j^+ - Z_j^-}{2} t \right), \\ & C_j(t) = \exp
\left( H_2 t \right), & D_{i \, j}(t) = \exp \left( M_{ij} t \right).
\end{eqnarray*}
The Lorentzian metric on $\mathbb{M}^{n-1}$ induces the following
non-degenerate bilinear form on $\mathfrak{n}_{n}^{-}$:
\begin{multline}\label{ICKEDEG}
\langle x_1 Q_1^- + x_2 Q_2^- + \sum_{i=3}^{n-1} x_{i} Y_{i-2}^- ,
y_1 Q_1^- + y_2 Q_2^- + \sum_{i=3}^{n-1} y_{i} Y_{i-2}^-   \rangle \\
= - x_1 y_1 + x_2 y_2 + \sum_{i=3}^{n-1} x_{i} y_{i}.
\end{multline}
We prove that for any $Y \in \mathfrak{n}_n^-$ such that $\langle Y,Y
\rangle = 1$ there exists some $m \in M_0^n$ satisfying
\begin{equation}
\Ad(m) Q_2^- =  Y.
\label{slutf}
\end{equation}
Let $Y = x_1 Q_1^- + x_2 Q_2^- + x_3 Y_1^- +
\dots + x_{n-1} Y_{n-3}^-$. For some $t \in \mathbb{R}$ we have
\begin{equation}
\Ad \left( \exp( t H_2 ) \right ) Q_2^- = Q_2^- \cosh t + Q_1^- \sinh t =
Q_2^- \cosh t + x_1 Q_1^-.
\label{INV_1}
\end{equation}
Using
\begin{equation}
\begin{split} \label{ADBER}
\Ad \left( B_j(t) \right) \left(
Q_2^- \sin \alpha  + Y_j^- \cos \alpha  \right) & =
Q_2^- \sin(\alpha + t) + Y_j^- \cos(\alpha + t), \\
\Ad \left( \exp( M_{i j} t ) \right) \left( Y_j^- \cos \alpha + Y_i^- \sin
\alpha \right) &= Y_j^- \cos(\alpha+t) + Y_i^- \sin(\alpha+t),
\end{split}
\end{equation}
we see that $\Ad \left( B_j(t) \right)$ and $\Ad \left( D_{ij}(t)
\right)$ acts as rotations in the subspace spanned by $Q_2^-, Y_1^- ,
\dots, Y_{n-3}^-$ leaving the subspace spanned by $Q_1^-$ invariant,
and \eqref{slutf} follows. We conclude that
\begin{enumerate}
\item $\Ad(m) (\mathfrak{n}_n^-) \subset 
\mathfrak{n}_n^- $ for all $m \in M_0^n$,
\item $\langle X,Y \rangle = \langle \Ad(m)X, \Ad(m)Y \rangle$ for all $X,Y \in 
\mathfrak{n}_n^- $, $m \in  M_0^n$,
\item $\Ad(M_0^n)$ acts transitively on the set $\{ X \in \mathfrak{n}^-_n \,
\vert \, X \neq 0, \, \langle X,X \rangle=c \}$ for each real number $c > 0$.
\end{enumerate}
In the following we choose 
\begin{equation}
X_1 = Q_1^-, X_2 = Q_2^-, X_3=Y_1^-, \dots, X_{n-1} =
Y_{n-3}^-
\end{equation} 
as a basis for $\mathfrak{n}_n^-$. Now assume that $[X,P]=0$ for all
$X \in \mathfrak{m}_n$. This implies that $\Ad(m)P=P$ for all $m \in
M_0^n$. Let $P=\sum a_{r_1 \dots r_{n-1}} X_1^{r_1} X_2^{r_2} \dots
X_{n-1}^{r_{n-1}}$. We define $P^*: \mathfrak{n}^-_n \rightarrow
\mathbb{C}$ by
\begin{equation}
P^{*}(\sum_{i=1}^{n-1} x_i X_i ) = \sum a_{r_1 \dots r_{n-1}}
x_1^{r_1} x_2^{r_2} \dots x_{n-1}^{r_{n-1}}.
\label{PSTAR}
\end{equation}
The Lie algebra $\mathfrak{n}^-_n$ is abelian, and hence
\begin{equation}\label{nabel}
(PQ)^*(\sum x_i X_i) = P^*(\sum x_i X_i) Q^*(\sum x_i X_i),
\end{equation}
which implies 
\begin{equation}\label{rot_for}
P^{*} (\Ad(m)^{-1} X) = (\Ad(m)P)^{*}(X).
\end{equation}
Using $\Ad(m)P=P$ for all $m \in M_0^n$ and \eqref{rot_for}, we
conclude that $P^{*}( \Ad(m)X)=P^{*}(X)$ for all $m \in M_0^n$ and $X
\in \mathfrak{n}^-_n$. Using (2) and (3) above, we see that for each
real number $c > 0$ the function $P^{*}$ is constant on $B_c := \{ X
\in \mathfrak{n}^-_n \, \vert \, \langle X, X \rangle = c\}$.

We get $P^{*}(x_2 X_2) = P^{*}(-x_2 X_2)$ and hence for some numbers
$a_k$ we have
\begin{equation*}
P^{*}(x_2 X_2) = \sum_{k=0}^{N} a_k (x_2)^{2k} = \sum_{k=0}^{N} a_k \langle
x_2 X_2, x_2 X_2 \rangle^{k}.
\end{equation*}
Using that for $c > 0$ the function $P^*$ is constant on $B_c$, we conclude that
for $X \in \mathfrak{n}_{n}^-$, $\langle X,X \rangle > 0 $
we have
\begin{equation}
P^{*}(X) = \sum_{k=0}^{N} a_k \langle X,X \rangle^k.
\label{lik}
\end{equation}
If we let $X = \sum_{i=1}^{n-1} x_i X_i$ and use \eqref{PSTAR},
together with $\langle X,X \rangle = -x_1^2 + x_2^2 + \dots +
x_{n-1}^2$, we get
\begin{equation}
 \sum a_{r_1 \dots r_{n-1}}
x_1^{r_1} x_2^{r_2} \dots x_{n-1}^{r_{n-1}} = \sum_{k=0}^{N} a_k (-x_1^2 +
x_2^2 + \dots + x_{n-1}^2)^k \,\, 
\end{equation}
on the open set $\{ (x_1, x_2, \dots, x_{n-1}) \in \mathbb{R}^{n-1} \,
\vert \, -x_1^2 + x_2^2 + \dots + x_{n-1}^2 > 0 \}$. But since two
polynomials coincide if they coincide on a non-empty open set,
\eqref{lik} is valid for all $X \in \mathfrak{n}_n^-$. We conclude
that
$$
P = \sum_{k=0}^N a_k (-X_1^2 + \dots + X_{n-1}^2)^k.
$$
\end{proof}

The following result will be used to construct the induced operator
families. It will be convenient to use the notation $T_{n} =
Y_{n-2}^{-} \in \mathfrak{n}^{-}_{n+1}$.

\begin{theorem}\label{alg_core} Let $\lambda \in \mathbb{C}$.
\begin{equation}
\mathcal{D}_{2N}(\lambda) = \sum_{j=0}^{N} a_{j}(\lambda)
(\LAP_{n-1})^j (T_{n})^{2N-2j} \in \mathcal{U}(\mathfrak{n}^-_{n+1})
\label{alg_core_form}
\end{equation}
satisfies
\begin{equation}\label{ekvation_att_visa}
[X,\mathcal{D}_{2N}(\lambda) ] \in \ARBIT
\end{equation}
for all $X \in \mathfrak{n}_n^+$ if and only if
\begin{equation}\label{koeff_villkor}
(N-j+1)(2N-2j+1) a_{j-1}(\lambda) + j(n-1+2\lambda-4N+2j) a_j(\lambda) = 0
\end{equation}
for $j=1,\ldots,N$. 
\end{theorem}

\begin{theorem}\label{alg_core2} Let $\lambda \in \mathbb{C}$. 
$$
\mathcal{D}_{2N+1}(\lambda) = \sum_{j=0}^{N} b_j(\lambda) (\LAP_{n-1})^{j}
(T_n)^{2N-2j+1}  \in \mathcal{U}(\mathfrak{n}^-_{n+1})$$ satisfies
\begin{equation*}
[X, \mathcal{D}_{2N+1}(\lambda) ] \in \ARBIT
\end{equation*}
for all $X \in \mathfrak{n}_n^+$ if and only if
\begin{equation}\label{koeff_villkor2}
(N-j+1)(2N-2j+3) b_{j-1}(\lambda) + j(n-3+2\lambda-4N+2j) b_j(\lambda) = 0
\end{equation}
for all $j=1, \dots, N$.
\end{theorem}


Let $\mathcal{D}_{2N}^0(\lambda) \in \mathcal{U}(\mathfrak{n}_{n+1}^{-})$ 
be the unique element satisfying
\eqref{koeff_villkor} and $a_{N}(\lambda)=1$. Let
$\mathcal{D}_{2N+1}^0(\lambda) \in
\mathcal{U}(\mathfrak{n}_{n+1}^{-})$ be the unique element satisfying
\eqref{koeff_villkor2} and $b_{N}(\lambda)=1$.  We have the following
algebraic characterization of the family $\mathcal{D}_{N}^0(\lambda)$.

\begin{theorem}\label{thm_unikhet} Let $N \geq 0, \mathcal{E} \in 
\mathcal{U}(\mathfrak{n}^{-}_{n+1})$ and let $\lambda \in \mathbb{C}$.
The three conditions
\begin{align}
[X,\mathcal{E}] & = 0 \; \text{for all} \; X \in \mathfrak{m}_n,
\label{villkorA} \\ [X, \mathcal{E}] &\in \ARBIT \; \text{for all} \;
X \in \mathfrak{n}_n^+,
\label{villkorB} \\
[H, \mathcal{E}] &= -N \mathcal{E}, \label{villkorC}
\end{align}
are satisfied if and only if $\mathcal{E} = c
\mathcal{D}_{N}^{0}(\lambda)$ for some $c \in \mathbb{C}$.
\end{theorem}

The analogs of Theorem \ref{alg_core} and Theorem \ref{alg_core2}
for $\mathfrak{o}(1,n)$ have been proved in \cite{juhl_conform}. Our
proof for $\mathfrak{o}(2,n-1)$ is similar but incorporates some
simplifications (as a proof of Lemma \ref{indlemma} using an induction
argument). We now prove Theorem \ref{alg_core} and Theorem
\ref{thm_unikhet}.  We omit the proof of Theorem \ref{alg_core2};
it is similar to the proof of the corresponding result for
$\mathfrak{o}(1,n)$ in \cite{juhl_conform}.

We formulate some lemmas which are used in the proof of Theorem \ref{alg_core}.

\begin{lemma}\label{H_lemma}\label{M_bracket}
We have
\begin{align}
[\mathfrak{m}_n, \LAP_{n-1}] & = 0, \label{mkomm1} \\
 [\mathfrak{n}_{n}^{-}, \LAP_{n-1} ] &= 0, \label{mkomm1b} \\
[\mathfrak{m}_n, T_n ] &= 0, \label{mkomm2} \\
[ H, (\LAP_{n-1})^k ] & = -2k (\LAP_{n-1})^k, \label{akomm} \\
[H, (T_n)^{2k} ] &= -2k (T_n)^{2k}.
\label{H1YB} 
\end{align}
\end{lemma}

\begin{proof} \eqref{mkomm1} follows from Proposition \ref{kommprop},
\eqref{mkomm1b} follows because $\mathfrak{n}_{n}^-$ is abelian. The
other statements follow from a simple computation.
\end{proof}

\begin{lemma}\label{formlemma} For any $X \in \mathfrak{n}_n^+$ and 
for any $F \in \mathcal{U}(\mathfrak{n}_{n+1}^{-})$ we have that
\begin{equation}
[X,F] \in (\mathcal{U}(\mathfrak{n}_{n+1}^-)  \mathfrak{m}_{n+1}) 
\oplus (\mathcal{U}(\mathfrak{n}_{n+1}^-)  \mathfrak{a})
\oplus \mathcal{U}(\mathfrak{n}_{n+1}^-). 
\label{stat}
\end{equation}
\end{lemma}

\begin{proof} Any $F \in \mathcal{U}(\mathfrak{n}_{n+1}^{-})$ can be written
as a linear combination of terms of the form
\begin{equation*}
(I_{1} )^{a_1}  \dots (I_{n} )^{a_{n}},
\end{equation*}
where $a_{1}, \dots, a_{n}$ are natural numbers and $I_1, \dots, I_n$ is a
basis for $\mathfrak{n}_{n+1}^-$; see \eqref{NBASIS}. Using 
$$ 
[A,B_1 \dots B_n] = [A,B_1] B_2 \dots B_n + \dots + B_1 \dots
B_{n-1} [A, B_n]
$$ and $[ \mathfrak{n}_{n+1}^{+}, \mathfrak{n}_{n+1}^- ] \subset
\mathfrak{m}_{n+1} \oplus \mathfrak{a}$, we conclude that $[X,F]$ can
be written as a linear combination of terms which are monomials with
factors from $\mathfrak{n}_{n+1}^-$ and at most one factor from
$\mathfrak{m}_{n+1} \oplus \mathfrak{a}_{n+1}$. Using
$[\mathfrak{m}_{n+1}, \mathfrak{n}^{-}_{n+1} ] \subset
\mathfrak{n}^{-}_{n+1}$ and $[\mathfrak{a}, \mathfrak{n}^{-}_{n+1} ]
\subset \mathfrak{n}^{-}_{n+1}$, the statement \eqref{stat} follows.
\end{proof}

\begin{lemma}\label{juhl_lemma}
\begin{equation*}
[ Y_1^+, (Y_1^-)^{2N} ] = -2N(2N-1)(Y_1^-)^{2N-1} +
4N(Y_{1}^-)^{2N-1} H.
\end{equation*}
For $2 \leq r \leq n$ we have
\begin{equation*}
[ Y_1^+, (Y_r^-)^{2N} ] = 2N(2N-1) Y_{1}^- 
(Y_{r}^-)^{2N-2} + 4N(Y_r^-)^{2N-1} M_{1r}. \end{equation*}
\end{lemma}

\begin{proof} See \cite{juhl_conform}. \end{proof}

\begin{lemma}\label{indlemma_HJALP}
$$ 
[Y_1^+, \ADSL] = (2n-6) Y_{1}^- + 4 Y_{1}^- H + v,
$$
where  $v \in \mathfrak{n}_n^- \otimes \mathfrak{m}_n$.
\end{lemma}

\begin{proof} We have
\begin{equation*}
\begin{split}
[Y_{1}^+, -(Q_1^-)^2 ] &= 2Q_1^- (Z_1^+ + Z_1^-) + 2Y_1^-, \\
[Y_{1}^+, (Q_2^-)^2 ] &= 2Q_2^- (Z_1^- - Z_1^+) + 2Y_1^-,
\end{split}
\end{equation*}
and hence
$$ 
[Y_1^+, -(Q_1^-)^2 + (Q_2^-)^2] = 4Y_1^- + 2Q_1^- (Z_1^+ + Z_1^-) +
2Q_2^- (Z_1^- - Z_1^+).
$$
Furthermore
$$ 
[Y_1^+, \sum_{j=1}^{n-3} (Y_j^{-})^2 ] = (2n-10)Y_{1}^- + 4Y_1^-
H + \sum_{j=2}^{n-3} 4Y_j^- M_{1 j}.
$$
\end{proof}

\begin{lemma} For $j \geq 1$ we have
\begin{equation}\label{ind_formel}
[Y_{1}^+, (\ADSL)^j ] = 2j(n-1-2j)Y_{1}^- (\ADSL)^{j-1} + 4j Y_{1}^{-}
(\ADSL)^{j-1}H + v,
\end{equation}
where $v \in \mathcal{U}(\mathfrak{n}_n^-) \otimes \mathfrak{m}_n$.
\label{indlemma}
\end{lemma}

\begin{proof}
Lemma \ref{indlemma_HJALP} shows that \eqref{ind_formel} is true for
$j=1$. Assume that \eqref{ind_formel} is valid for $j \leq k$. We want to prove that
\eqref{ind_formel} is valid for $j=k+1$. We get for some $v_1, v_2 \in
\mathcal{U}(\mathfrak{n}_{n}^-) \otimes \mathfrak{m}_n$ that
\begin{equation*}
\begin{split}
& [Y_1^+, (\ADSL)^{k+1} ] = [Y_1^+, \ADSL]
(\ADSL)^{k} + (\ADSL) [Y_1^+, (\ADSL)^k] \\
&= \{ (2n-6) Y_1^- + 4Y_1^- H + v_1 \} (\ADSL)^k \\
& + (\ADSL) \{ 2k(n-1-2k) Y_1^- (\ADSL)^{k-1} +4k
Y_1^- (\ADSL)^{k-1} H + v_2 \}.
\end{split}
\end{equation*}
From \eqref{akomm} we see that
\begin{equation*}
4Y_1^- H (\ADSL)^k = 4Y_1^- (\ADSL)^k (H - 2k).
\end{equation*}
From \eqref{mkomm1} we know that $[\mathfrak{m}_n,\ADSL]=0$ and hence
for all $R \in \mathcal{U}(\mathfrak{n}_n^-)$, $S \in \mathfrak{m}_n$
we have
\begin{equation*}
R S (\ADSL)^k = R (\ADSL)^k S = R^{'} S,
\end{equation*}
where $R' \in \mathcal{U}(\mathfrak{n}_n^-)$. 

We conclude that for some $v_3, v_4 \in \mathcal{U}(\mathfrak{n}_n^-)
\otimes \mathfrak{m}_n$ we have
\begin{equation*}
\begin{split}
[Y_{1}^+, (\ADSL)^{k+1} ] &= (2n-6) Y_{1}^{-} (\ADSL)^k + 4Y_{1}^-
(\ADSL)^k (H - 2k) + v_3 \\
&+ 2k(n-1-2k) Y_{1}^{-} (\ADSL)^{k} + 4k Y_{1}^- (\ADSL)^k H + v_4 \\
&= 2(k+1)(n-1-2(k+1)) Y_{1}^{-} (\ADSL)^k \\
&+ 4(k+1) Y_{1}^- (\ADSL)^{k} H
+ (v_3+v_4).
\end{split}
\end{equation*}
\end{proof}

\begin{proof}[Proof of Theorem \ref{alg_core}] For some $v \in 
\mathcal{U}(\mathfrak{n}_n^-) \otimes \mathfrak{m}_n$ we get using
 $[\mathfrak{m}_n, T_n] = 0$ that
\begin{equation*}
\begin{split}
&Y_{1}^{+} (\ADSL)^j (T_n)^{2k} = [Y_{1}^+, (\ADSL)^j] (T_n)^{2k} +
(\ADSL)^{j} Y_{1}^+ (T_n)^{2k} \\ &= 2j(n-1-2j)Y_{1}^{-} (\ADSL)^{j-1}
(T_n)^{2k} \\ &+4jY_{1}^- (\ADSL)^{j-1} H (T_n)^{2k} + v (T_n)^{2k} +
(\ADSL)^{j} Y_{1}^{+} (T_n)^{2k} \\ & = 2j(n-1-2j)Y_{1}^{-}
(\ADSL)^{j-1} (T_n)^{2k} \\&+ 4jY_{1}^- (\ADSL)^{j-1} H (T_n)^{2k} + w +
(\ADSL)^{j} Y_{1}^{+} (T_n)^{2k},
\end{split}
\end{equation*}
where $w \in \mathcal{U}(\mathfrak{n}_{n+1}^-) \otimes
\mathfrak{m}_n$. Using $H (T_n)^{2k} = (T_n)^{2k}(H-2k)$, we get
\begin{equation*}
\begin{split}
&[ Y_{1}^{+}, (\ADSL)^j (T_n)^{2k}] \\ &= 2j(n-1-2j)Y_{1}^{-}
(\ADSL)^{j-1} (T_n)^{2k} \\ &+ 4jY_{1}^- (\ADSL)^{j-1} H (T_n)^{2k} + w +
(\ADSL)^{j} [Y_{1}^+, (T_n)^{2k} ] \\ &= 2j(n-1-2j)Y_{1}^{-}
(\ADSL)^{j-1} (T_n)^{2k} \\ &+ 4jY_{1}^- (\ADSL)^{j-1} (T_n)^{2k} H - 8jk
Y_{1}^- (\ADSL)^{j-1} (T_n)^{2k} + w \\ &+ (\ADSL)^{j} 2k(2k-1)
Y_{1}^- (T_n)^{2k-2} + (\ADSL)^j 4k(T_n)^{2k-1} S, \,\,\,\,\,
(\text{where} \,\, S \in \mathfrak{m}_{n+1})\\ &= \{
2j(n-1-2j)+4j\lambda - 8jk \} Y_{1}^- (\ADSL)^{j-1} (T_n)^{2k} \\ &+
2k(2k-1) Y_{1}^- (\ADSL)^{j} (T_n)^{2k-2} \quad \text{mod} \,
\mathcal{U}(\mathfrak{n}_{n+1}^-) (\mathfrak{m}_{n+1} \oplus
\mathbb{C}(H-\lambda)),
\end{split}
\end{equation*}
where we used that $[Y_{1}^-, \ADSL]=0$, that is \eqref{bracket_rel}.

We obtain (mod $\mathcal{U}(\mathfrak{n}_{n+1}^-)(\mathfrak{m}_{n+1}
\oplus \mathbb{C}(H-\lambda))$)
\begin{equation*}
\begin{split}
&[Y_{1}^+, \mathcal{D}_{2N}(\lambda)] \\ =&\sum_{j=0}^{N} a_j(\lambda) (
(2j(n-1-2j) + 4j\lambda - 8j(N-j) ) Y_{1}^- (\ADSL)^{j-1}
(T_n)^{2N-2j} \\ &+ \sum_{j=0}^{N} 2(N-j)(2N-2j-1) Y_{1}^- (\ADSL)^j
(T_n)^{2N-2j-2}.
\end{split}
\end{equation*}
We conclude that $[Y_{1}^+, \mathcal{D}_{2N}(\lambda) ] = 0$ (mod $\ARBIT$)
if and only if
\begin{equation*}
a_r(\lambda) \{ 2r(n-1-2r) + 4r\lambda - 8r(N-r) \} + a_{r-1}(\lambda)
2(N-r+1)(2N-2r+1) = 0
\end{equation*}
for $r=1, \ldots, N$. 

The \emph{only if} statement follows because no
non-trivial linear combination of
$$Y_{1}^{-} (\ADSL)^{k-1} (T_{n}^-)^{2N-2k} = Y_{1}^{-} (\ADSL)^{k-1}
(T_{n}^-)^{2N-2k} \in \mathcal{U}(\mathfrak{g}_{n+1})$$ for $k=1,
\ldots, N$ lies in $\ARBIT$. This follows expanding $(\ADSL)^{k-1}$
into monomials and using the Poincaré-Birkhoff-Witt theorem.

Now suppose that the coefficients $a_j(\lambda)$ satisfy
\eqref{koeff_villkor}. Under this assumption, we know that
\eqref{ekvation_att_visa} is true for $X=Y_{1}^+$, that is, we know
that
\begin{equation*}
[Y_{1}^+, \mathcal{D}_{2N}(\lambda) ] \in \ARBIT.
\end{equation*}
We prove that this implies \eqref{ekvation_att_visa} for \emph{all} $X
\in \mathfrak{n}^+_n$.

First, observe that using \eqref{bracket_rel}, we see that for all
$\tilde{M} \in M^{n}_0$ we have
\begin{equation}\label{ADS_komm}
\Ad(\tilde{M}) \ARBIT \subset \ARBIT.
\end{equation}

Using $[\mathfrak{m}_n, \ADSL] = 0$ and $[\mathfrak{m}_n, T_n] = 0$,
we get
\begin{equation*}
\begin{split}
\ARBIT &\ni
\Ad(\tilde{M}) [Y_1^+, \mathcal{D}_{2N}(\lambda) ] \\&= [\Ad(\tilde{M}) Y_1^+,
  \Ad(\tilde{M}) \mathcal{D}_{2N}(\lambda) ] \\
&= [ \Ad(\tilde{M}) Y_1^+, \mathcal{D}_{2N}(\lambda) ] \\
&= \begin{cases} [ Y_{j}^+, \mathcal{D}_{2N}(\lambda) ] \,\,\,\,\quad\quad \text{for} \;
  \tilde{M}=\exp( \frac{-\pi M_{1j}}{2} ), \\
[Y_1^+ + Q_2^+ - Q_1^+,  \mathcal{D}_{2N}(\lambda) ] \,\,\, \text{for} \;
  \tilde{M} = Z_1^+, \\
[Y_1^+ -Q_1^+ - Q_2^+,  \mathcal{D}_{2N}(\lambda) ] \,\,\, \text{for} \; \tilde{M} =
  Z_1^-. \\
 \end{cases}
\end{split}
\end{equation*}
\end{proof}

\begin{proof}[Proof of Theorem \ref{thm_unikhet}] We only give the proof for 
the case of even $N$. Assume $\mathcal{E} \in
\mathcal{U}(\mathfrak{n}^{-}_{n+1})$ satisfies \eqref{villkorA},
\eqref{villkorB} and \eqref{villkorC}. Using the notation
\eqref{NBASIS}, the Poincaré-Birkhoff-Witt theorem shows that
$\mathcal{E}$ can be written as a linear combination of
\begin{equation}
\{ (I_1)^{a_1} \dots (I_n)^{a_n} \; \; \vert \; \; a_1, \dots, a_n \in
\mathbb{N} \}.
\label{BASIS_EL}
\end{equation}
Now \eqref{villkorC} and the fact that $[H, X] = -X$ for
all $X \in \mathfrak{n}_{n}^{-}$ show that $\mathcal{E}$ is a linear
combination of basis elements from \eqref{BASIS_EL}, where
\begin{equation}
a_1 + \dots + a_n = 2N.
\label{deg_villkor}
\end{equation}
We see that
\begin{equation*}
\mathcal{E} = \sum_{j=0}^{2N} p_j T_n^{2N-j},
\end{equation*}
where $p_j \in \mathcal{U}(\mathfrak{n}_n^-)$ is a homogeneous
polynomial of degree $j$ in the variables $I_1, \dots, I_{n-1}$ given by \eqref{NBASIS}.

For all $R \in \mathfrak{m}_n$ we have $[R, T_n] = 0$  and hence
\begin{equation*}
[R, \mathcal{E}] = [R,p_1] T_{n}^{2N-1} + \dots + [R, p_{2N-1}]
T_{n} + [R,p_{2N}] = 0.
\end{equation*}
We have $[\mathfrak{m}_n, \mathfrak{n}_n^-] \subset \mathfrak{n}_n^-$ and
hence
$[R, p_j ] \in \mathcal{U}(\mathfrak{n}_n^-)$. Using the
Poincaré-Birkhoff-Witt theorem, we conclude that for $j=1, \dots, 2N$ we
have that $[R,p_j] = 0$.

Proposition \ref{kommprop} shows that $p_j$ is a polynomial in $\ADSL$
for every $j=1, \dots, 2N$. Using that $p_j$ is a homogeneous
polynomial of degree $j$, we deduce that $p_j = 0$ for odd $j$, and
$p_{2i} = c_{i}\left( \ADSL \right)^i$ for some complex constant
$c_i$.

We conclude that
\begin{equation*}
\mathcal{E} = \sum_{j=0}^{N} a_j(\lambda) (\LAP_{n-1})^{j} (T_n)^{2N-2j},
\end{equation*}
and we see that $\mathcal{E}$ is of the form
\eqref{alg_core_form}. Now the \emph{only if} statement in Theorem
\ref{alg_core} completes the proof.
\end{proof}

\section{$\mathcal{D}_{N}^{0}(\lambda)$ as families of homomorphisms of 
Verma modules}\label{verma}

Let $W$ be a $\mathcal{U}(\mathfrak{p_m})$-module. The algebra
$\mathcal{U}(\mathfrak{g}_m)$ acts on the vector space
$\mathcal{U}(\mathfrak{g}_m ) \otimes W$ by
\begin{equation}\label{u_verkan_formel}
u_1 (u_2 \otimes w) = (u_1 u_2) \otimes w.
\end{equation}
Let $I_{W}(\mathfrak{g}_m) \subset \mathcal{U}(\mathfrak{g}_m) \otimes
W$ be the left $\mathcal{U}(\mathfrak{g}_m)$-ideal generated by the
elements $X \otimes w - 1 \otimes (X \cdot w)$, $X \in
\mathfrak{p}_m$. $I_{W}(\mathfrak{g}_m)$ equals the subspace of
$\mathcal{U}(\mathfrak{g}_m) \otimes W$ spanned by the elements
\begin{equation}\label{spannedE}
(n p) \otimes w - n \otimes (p \cdot w),
\end{equation}  
where $n \in \mathcal{N}, p \in \mathcal{P}, w \in W$. Here $\mathcal{N}$ is a basis
for $\mathcal{U}(\mathfrak{n}^-_m)$ and $\mathcal{P}$ is a basis for  
$\mathcal{U}(\mathfrak{p}_m)$. In fact, for any $u \in
\mathcal{U}(\mathfrak{g}_{m})$, $p \in \mathcal{U}(\mathfrak{p}_{m})$
and $w \in W$ 
\begin{equation*}
u p \otimes w - u \otimes (p \cdot w)
\end{equation*}
can be written as a linear combination of terms of the form
\eqref{spannedE}. Using the Poincaré-Birkhoff-Witt theorem, we know
that $u$ can be written as a linear combination of terms of the form
$n_0 p_0$ where  $n_0 \in \mathcal{U}(\mathfrak{n}^{-}_{m})$, $p_0 \in
\mathcal{U}(\mathfrak{p}_m)$. We see that
\begin{multline*}
n_0 p_0 p \otimes w - n_0 p_0 \otimes (p \cdot w) \\
= \Big( n_0 p_0 p
\otimes w - n_0 \otimes (p_0 p) \cdot w \Big) - \Big(n_0 p_0 \otimes
(p \cdot w) + n_0 \otimes (p_0 \cdot (p \cdot w)) \Big).
\end{multline*}

The \emph{generalized Verma module} induced from $W$ is defined by
the vector space quotient
\begin{equation}\label{GVM}
\mathcal{M}_{W} (\mathfrak{g}_m) = (\mathcal{U}(\mathfrak{g}_m)\otimes W)/
I_{W}(\mathfrak{g}_m)
\end{equation}
with the left $\mathcal{U}(\mathfrak{g}_m)$-action. We also use the
notation $\mathcal{U}(\mathfrak{g}_m) \otimes_{\mathcal{U}(\mathfrak{p}_m)} W$ for
$\mathcal{M}_{W}(\mathfrak{g}_m)$. 

For $\lambda \in \c$ let $\xi_{\lambda}: \mathfrak{a} \rightarrow
\mathbb{C}$ be the character $\xi_{\lambda}(tH) = t\lambda$. The
vector space $W = \mathbb{C}_{\lambda}$ is made into a left
$\mathcal{U}(\mathfrak{p_m})$-module by
\begin{align*}
X z & = \xi_{\lambda}(X) z & \text{for} \; X \in \mathfrak{a}, \\
X z & = 0 & \text{for} \; X \in \mathfrak{m}_m \oplus \mathfrak{n}^{+}_m.
\end{align*}
Let 
$$ 
I_{\lambda}(\mathfrak{g}_m) 
= I_{\mathbb{C}_{\lambda}}(\mathfrak{g}_m) \quad \mbox{and}
\mathcal{M}_{\lambda}(\mathfrak{g}_m) =
\mathcal{M}_{\mathbb{C}_{\lambda}}(\mathfrak{g}_m).
$$

We view $\mathcal{U}(\mathfrak{g}_{n+1}) \otimes \mathbb{C}_{\lambda}$
as a $\mathcal{U}(\mathfrak{g}_n)$-module with the action
$$
u_1(u \otimes v) = (i(u_1) u) \otimes v
$$ 
for $u_1 \in \mathcal{U}(\mathfrak{g}_n)$, where 
$i: \mathcal{U}(\mathfrak{g}_n) \rightarrow \mathcal{U}(\mathfrak{g}_{n+1})$ 
is the inclusion induced by \eqref{inkl}.

\begin{lemma}\label{easylemma} $\lambda \in \mathbb{C}$, 
$W = \mathbb{C}_{\lambda}$, $u \otimes 1 \in I_{W}(\mathfrak{g}_m)$,
$u \in \mathcal{U}(\mathfrak{n}_{m}^{-})$ implies that $u=0$.
\end{lemma}

\begin{proof} Observe that any element in $I_{\mathbb{C}_{\lambda}}(\mathfrak{g}_m)$ is
spanned by elements of the form \eqref{spannedE}, and hence $u \otimes
1$ can be written as a linear combination of
\begin{equation}
(n_i p_i) \otimes 1 - \xi_{\lambda}(p_i) n_i \otimes 1 = ( n_i p_i -
\xi_{\lambda}(p_i) n_i ) \otimes 1, \,\, i=1, \dots, r
\label{easyp}
\end{equation}
for some elements $p_1, p_2, \dots, p_r \in \mathcal{P}$ and elements
$n_1, \dots, n_r \in \mathcal{N}$ such that $n_i p_i$, $i=1, \dots, r$
are linearly independent. We conclude that $u$ is a linear
combination of
\begin{equation}
n_i p_i - \xi_{\lambda}(p_i) n_i, \,\,\, i=1, \dots, r.
\label{easyp2}
\end{equation}
Now assume that $p_i \not \in \mathbb{C}$ for all $1 \leq i \leq r$. 
We see that vector space spanned by \eqref{easyp2} has empty
intersection with $\mathcal{U}(\mathfrak{n}_m^-)$. This implies that
$u = 0$. 

Now assume that $p_i \in \mathbb{C}$ for some $1 \leq i \leq
r$. Reordering the terms in \eqref{easyp2}, we can assume that $p_1,
\dots, p_s \in \mathbb{C}$, and $p_{s+1}, \dots, p_r \not \in
\mathbb{C}$. Now $p_i \in \mathbb{C}$ for $1 \leq i \leq s$ implies
that $n_i p_i - \xi_{\lambda}(p_i) n_i = 0$ for $1 \leq i \leq s$. We
conclude that $u$ is a linear combination of \eqref{easyp2} for $s+1
\leq i \leq r$. But the vector space spanned by the terms in
\eqref{easyp2} with $s+1 \leq i \leq r$ has empty intersection with
$\mathcal{U}(\mathfrak{n}_m^-)$, and hence $u=0$. \end{proof}

\begin{theorem} For any $\lambda \in \c$, the map
\begin{equation}\label{huvudavbildning}
\mathcal{U}(\mathfrak{g}_n) \otimes \mathbb{C}_{\lambda-N} \ni T
\otimes 1 \mapsto i(T) \mathcal{D}_N^{0}(\lambda) \otimes 1 \in
\mathcal{U}(\mathfrak{g}_{n+1}) \otimes \mathbb{C}_{\lambda}
\end{equation}
induces a homomorphism
\begin{equation}
\varphi_0: \mathcal{M}_{\lambda-N}(\mathfrak{g}_n) \rightarrow
\mathcal{M}_{\lambda}(\mathfrak{g}_{n+1})
\label{homovar}
\end{equation}
of $\mathcal{U}(\mathfrak{g}_n)$-modules. Furthermore, 
$$ 
\mathrm{Hom}_{\mathcal{U}(\mathfrak{g}_n)}
(\mathcal{M}_{\lambda-N}(\mathfrak{g}_n), 
\mathcal{M}_{\lambda}(\mathfrak{g}_{n+1})) 
= \mathrm{span}_{\mathbb{C}}(\varphi_0 ).
$$
\end{theorem}

\begin{proof}
First, we prove that $\mathrm{Hom}_{\mathcal{U}(\mathfrak{g}_n) } (
\mathcal{M}_{\lambda-N}(\mathfrak{g}_n),
\mathcal{M}_{\lambda}(\mathfrak{g}_{n+1}) ) \subset \mathrm{span}_{\mathbb{C}} (
\varphi_0 )$. Assume that $$\varphi \in \mathrm{Hom}_{\mathcal{U}(\mathfrak{g}_n) } (
\mathcal{M}_{\lambda-N}(\mathfrak{g}_n),
\mathcal{M}_{\lambda}(\mathfrak{g}_{n+1}) ).$$
We have
$$
\varphi( \{ 1 \otimes 1 \} ) = \{ F \otimes 1 \},
$$
for some $F \in \mathcal{U}(\mathfrak{g}_{n+1})$. Here we
use $\{ \}$ to denote equivalence classes. Since
\begin{equation*}
\varphi ( \{ u \otimes 1 \} ) = \varphi( u \{ 1 \otimes 1 \} ) = i(u)
\varphi( \{ 1 \otimes 1 \} ) = i(u) \{ F \otimes 1 \} = \{ (i(u)F) \otimes 1 \}
\end{equation*}
for all $u \in \mathcal{U}(\mathfrak{g}_n)$, we conclude that $\varphi$
is induced from a map 
\begin{equation}\label{induced}
\zeta: \mathcal{U}(\mathfrak{g}_n) \otimes \mathbb{C}_{\lambda-N} \ni
T \otimes 1 \mapsto i(T) F \otimes 1 \in
\mathcal{U}(\mathfrak{g}_{n+1}) \otimes \mathbb{C}_{\lambda},
\end{equation}
where $F \in \mathcal{U}(\mathfrak{g}_{n+1})$. Furthermore, we can
assume without loss of generality that $F \in
\mathcal{U}(\mathfrak{n}_{n+1}^-)$: By the Poincaré-Birkhoff-Witt
theorem $F \in \mathcal{U}(\mathfrak{g}_{n+1})$ can be written as a
linear combination of terms of the form $n p$, where $n \in
\mathcal{U}(\mathfrak{n}_{n+1}^{-})$ and $p \in
\mathcal{U}(\mathfrak{p}_{n+1})$. For $n \in U(\mathfrak{n}^-_{n+1})$
and $p \in U(\mathfrak{p}_{n+1})$ we have for some complex number C
that
$$
\{i(T) np \otimes 1\} = \{i(T) n \otimes p \cdot 1 \} = \{ i(T) n \otimes C \}
= \{i(T) (Cn) \otimes 1\},
$$
and hence we conclude that
$$\{i(T) F \otimes 1\} = \{i(T) \tilde{F} \otimes 1\}$$
for some $\tilde{F} \in U(\mathfrak{n}^-_{n+1})$.

It is also easy to see that a map $f: \mathcal{U}(\mathfrak{g}_n)
\otimes \mathbb{C}_{\lambda-N} \rightarrow
\mathcal{U}(\mathfrak{g}_{n+1}) \otimes
\mathbb{C}_{\lambda} $ induces a homomorphism
$\mathcal{M}_{\lambda-N}(\mathfrak{g}_n) \rightarrow
\mathcal{M}_{\lambda}(\mathfrak{g}_{n+1})$ if and only if
\begin{equation}\label{villkorFOR}
f( I_{\lambda-N}(\mathfrak{g}_n) ) \subset I_{\lambda}(\mathfrak{g}_{n+1}).
\end{equation}
This implies that $$\zeta(
I_{\lambda-N} ( \mathfrak{g}_n )) \subset
I_{\lambda}(\mathfrak{g}_{n+1}).
$$

We now find conditions for $F$ which are equivalent to that $\varphi$
induces a $\mathcal{U}(\mathfrak{g}_n)$-homomorphism \eqref{homovar}. Observe that
\begin{equation}\label{not1}
H \otimes 1 - 1 \otimes \xi_{\lambda-N}(H)1 = H \otimes 1 - 1 \otimes
(\lambda-N) \in I_{\lambda-N}(\mathfrak{g}_n).
\end{equation}
We compute
\begin{equation}\label{not2}
\begin{split}
\zeta( H \otimes 1 - 1 \otimes (\lambda-N) ) &= H F \otimes 1 - F
\otimes (\lambda-N) \\
&= \left( [H,F] + F H \right)\otimes 1 - F \otimes
(\lambda-N) \\ &= [H,F] \otimes 1 + F(H \otimes 1 - 1 \otimes
\xi_{\lambda}(H) 1) + N F \otimes 1.
\end{split}
\end{equation}
We see that
\begin{equation}\label{HAA}
([H,F] + NF) \otimes 1 \in I_{\lambda}(\mathfrak{g}_{n+1}).
\end{equation}
Now $F \in \mathcal{U}(\mathfrak{n}_{n+1}^{-})$ and hence $[H, F] \in
\mathcal{U}(\mathfrak{n}_{n+1}^{-})$. Using \eqref{HAA} and Lemma
\ref{easylemma}, we see that
\begin{equation}
\boxed{[H,F] + NF = 0.}
\label{VILL_A}
\end{equation}
For $M \in \mathfrak{m}_{n}$ we have
\begin{equation}
M \otimes 1 - 1 \otimes \xi_{\lambda-N}(M)1 = M \otimes 1 \in
I_{\lambda-N}(\mathfrak{g}_n).
\label{not3}
\end{equation}
We compute
\begin{equation}
\zeta( M \otimes 1 ) = M F \otimes 1 = FM \otimes 1 + ([M,F] \otimes 1).
\label{not4}
\end{equation}
We see that
\begin{equation}
[M,F] \otimes 1 \in I_{\lambda}(\mathfrak{g}_{n+1}).
\label{HAA2}
\end{equation}
Now $F \in \mathcal{U}(\mathfrak{n}_{n+1}^{-})$ and hence $[M,F] \in
\mathcal{U}(\mathfrak{n}_{n+1}^{-})$. Using \eqref{HAA2} and Lemma
\ref{easylemma}, we see that
\begin{equation}
\boxed{[M,F] = 0.}
\label{VILL_B}
\end{equation}
For $X \in \mathfrak{n}_{n}^{+}$ we have
\begin{equation}
X \otimes 1 - 1 \otimes \xi_{\lambda-N}(X) = X \otimes 1 \in
I_{\lambda-N}(\mathfrak{g}_n).
\label{not5}
\end{equation}
We compute
\begin{equation}
\zeta ( X \otimes 1) = XF \otimes 1 = [X,F] \otimes 1 + FX \otimes 1.
\label{not6}
\end{equation}
We see that
\begin{equation*}
[X,F] \otimes 1 \in I_{\lambda}(\mathfrak{g}_{n+1}).
\end{equation*}
Using the notation \eqref{NBASIS}, $F \in
\mathcal{U}(\mathfrak{n}_{n+1}^{-})$ and Lemma \ref{formlemma} shows that
\begin{equation*}
\begin{split}
[X,F] \otimes 1 &= \sum_{a_1, \dots, a_{n} \in \mathbb{N}} c_{a_1 \dots a_{n} }
(I_1)^{a_1} \dots (I_{n})^{a_{n}} \otimes 1 \\
&+  \sum_{a_1, \dots, a_{n} \in \mathbb{N}} d_{a_1 \dots a_{n} }
(I_1)^{a_1} \dots (I_{n})^{a_{n}} H \otimes 1 \\
&+  \sum_{a_1, \dots, a_{n}
  \in \mathbb{N}, 1 \leq f,g \leq n } e_{a_1 \dots a_{n}, f, g }
(I_1)^{a_1} \dots (I_{n})^{a_{n}} M_{f,g} \otimes 1.
\end{split}
\end{equation*}
Using 
\begin{equation*}
(I_1)^{a_1} \dots (I_{n})^{a_{n}} M_{f,g} \otimes 1 \in I_{\lambda}(\mathfrak{g}_{n+1}),
\end{equation*}
we get 
\begin{equation*}
\begin{split}
&\sum_{a_1, \dots, a_{n} \in \mathbb{N}} c_{a_1 \dots a_{n} }
(I_1)^{a_1} \dots (I_{n})^{a_{n}} \otimes 1 \\
&+  \sum_{a_1, \dots, a_{n} \in \mathbb{N}} d_{a_1 \dots a_{n} }
(I_1)^{a_1} \dots (I_{n})^{a_{n}} H \otimes 1 \in I_{\lambda}(\mathfrak{g}_{n+1}),
\end{split}
\end{equation*}
and hence
\begin{equation*}
\begin{split}
\sum_{a_1, \dots, a_{n} \in \mathbb{N}} ( c_{a_1 \dots a_{n} } -
\lambda d_{a_1 \dots a_{n}} ) (I_1)^{a_1} \dots (I_{n})^{a_{n}}
\otimes 1 \in I_{\lambda}(\mathfrak{g}_{n+1}).
\end{split}
\end{equation*}
Using Lemma \ref{easylemma}, we conclude that $c_{a_1 \dots a_{n}} =
\lambda d_{a_1 \dots a_{n}}$ and hence
\begin{equation}\label{VILL_C}
\boxed{[X,F] \in \mathcal{U}(\mathfrak{n}_{n+1}^-)(\mathfrak{m}_{n+1} 
\oplus \mathbb{C}(H-\lambda)).}
\end{equation}
Theorem \ref{thm_unikhet} and \eqref{VILL_A}, \eqref{VILL_B} and
\eqref{VILL_C} show that $X \in \mathrm{span}_{\mathbb{C}}( \varphi_0)$.

Second, we prove that \eqref{huvudavbildning} satisfy \eqref{villkorFOR}, 
which implies that 
$$ 
\mathrm{span}_{\mathbb{C}}(\varphi_0 ) \subset
\mathrm{Hom}_{\mathcal{U}(\mathfrak{g}_n)}
(\mathcal{M}_{\lambda-N}(\mathfrak{g}_n),
\mathcal{M}_{\lambda}(\mathfrak{g}_{n+1})).
$$ 
From Theorem \ref{thm_unikhet} we see that \eqref{villkorA},
\eqref{villkorB} and \eqref{villkorC} are satisfied for
$\mathcal{E}=\mathcal{D}_{N}^{0}(\lambda)$.

Using \eqref{not1}, \eqref{not2} and \eqref{villkorC}, we conclude that
$$ 
\varphi_{0}(H \otimes 1 - 1 \otimes \xi_{\lambda-N}(H)1 ) \in
I_{\lambda}(\mathfrak{g}_{n+1}).
$$ 
Using \eqref{not3}, \eqref{not4} and \eqref{villkorA}, we conclude that
$$ 
\varphi_{0}(X \otimes 1 - 1 \otimes \xi_{\lambda-N}(X)1) \in
I_{\lambda}(\mathfrak{g}_{n+1})
$$ 
for all $X \in \mathfrak{m}_n$. Using \eqref{not5}, \eqref{not6}
and \eqref{villkorB}, we conclude that
$$ 
\varphi_{0}(X \otimes 1 - 1 \otimes \xi_{\lambda-N}(X)1) \in
I_{\lambda}(\mathfrak{g}_{n+1})
$$ 
for all $X \in \mathfrak{n}_n^+$.
\end{proof}

\section{Induced families of differential intertwining operators}
\label{ind-fam}

In this section we use the algebraic results of Section \ref{algebra}
to prove equivariance of the polynomial families induced by
$\mathcal{D}_N^0(\lambda) \in \mathcal{U}(\mathfrak{n}_{n+1}^-)$ with
respect to certain principal series representations. We view $\mathcal{D}_{N}^{0}(\lambda) \in
\mathcal{U}(\mathfrak{g}_{n+1})$ as a left-invariant differential
operator $C^{\infty}(G^{n+1}) \rightarrow C^{\infty}(G^{n+1})$ using
$R$.
 
\begin{theorem}\label{HJA} Let $\lambda \in \mathbb{C}$. Then $\mathcal{E} \in
\mathcal{U}(\mathfrak{n}_{n+1}^{-})$ induces a left
$G^{n+1}$-equivariant map
\begin{equation*}
\Ind_{P_0^{n+1} }^{G^{n+1}}(\xi_{\lambda}) \rightarrow
\Ind_{P_0^{n} }^{G^{n+1}}(\xi_{\lambda-N})
\end{equation*}
if and only if $\mathcal{E} = c \mathcal{D}_{N}^{0}(\lambda)$
for some $c \in \mathbb{C}$.
\label{HUVA}
\end{theorem}

\begin{theorem} For $\lambda \in \mathbb{C}$ the element 
$\mathcal{E} \in \mathcal{U}(\mathfrak{n}_{n+1}^{-})$ induces a left
$G^{n+1}$-equivariant map
\begin{equation*}
\Ind_{P^{n+1}}^{G^{n+1}}(\xi_{\lambda}) \rightarrow
\Ind_{P^n}^{G^{n+1}}(\xi_{\lambda-N})
\end{equation*}
if and only if $\mathcal{E} = c \mathcal{D}_{N}^{0}(\lambda)$ for some
$c \in \mathbb{C}$ and $N$ is even.
\end{theorem} \label{HUVA_MOD}\label{s7-1}

\begin{theorem} For $\lambda \in \mathbb{C}$ the element $\mathcal{E} \in \mathcal{U}(\mathfrak{n}_{n+1}^{-})$ induces a left
$G^{n+1}$-equivariant map
\begin{equation*}
\Ind_{P^{n+1}}^{G^{n+1}}(\xi_{\lambda} ) \rightarrow
\Ind_{P^n}^{G^{n+1}}(\xi_{\lambda-N} \otimes \sigma_{-})
\end{equation*}
if and only if $\mathcal{E} = c \mathcal{D}_{N}^{0}(\lambda)$ for some
$c \in \mathbb{C}$ and $N$ is odd.
\label{SIGMA}
\end{theorem}

For even $N$, we define 
$$
D_N(\lambda): \Ind_{P^{n+1}}^{G^{n+1}}(\xi_{\lambda}) \rightarrow
\Ind_{P^n}^{G^{n}}(\xi_{\lambda-N})
$$ 
by $D_N(\lambda) = i^{*} \circ \widetilde{\mathcal{D}^{0}_N
(\lambda)}$, where $\widetilde{\mathcal{D}^0_N(\lambda)} :
C^{\infty}(G^{n+1}) \rightarrow C^{\infty}(G^{n+1})$ is the operator
induced by $\mathcal{D}_N^0(\lambda)$. 

For odd $N$, we define 
$$
D_N(\lambda): \Ind_{P^{n+1}}^{G^{n+1}}(\xi_{\lambda}) \rightarrow
\Ind_{P^n}^{G^{n}}(\xi_{\lambda-N} \otimes \sigma_{-})
$$
by $D_N(\lambda) = i^{*} \circ \widetilde{\mathcal{D}^{0}_N (\lambda)}$, 
where $\widetilde{\mathcal{D}^0_N(\lambda)}: C^{\infty}(G^{n+1}) \rightarrow
C^{\infty}(G^{n+1})$ is the operator induced by $\mathcal{D}_N^0(\lambda)$.

\begin{theorem}\label{HUVA2} For $\lambda \in \mathbb{C}$ the element  
$\mathcal{E} \in \mathcal{U}(\mathfrak{n}_{n+1}^{-})$ induces a left
$G^{n+1}$-equivariant map
\begin{equation*}
\Ind_{\PHAT^{n+1}}^{G^{n+1}}(\xi_{\lambda}) \rightarrow
\Ind_{\PHAT^n}^{G^{n+1}}(\xi_{\lambda-N})
\end{equation*}
if and only if $\mathcal{E} = c \mathcal{D}_{N}^{0}(\lambda)$ for some
$c \in \mathbb{C}$.
\end{theorem}

We define 
$$
\hat{D}_N(\lambda): \Ind_{\hat{P}^{n+1}}^{G^{n+1}}(\xi_{\lambda}) \rightarrow
\Ind_{\hat{P}^n}^{G^{n}}(\xi_{\lambda-N})
$$ by $\hat{D}_N(\lambda) = i^{*} \circ \widetilde{\mathcal{D}^{0}_N
(\lambda)}$, where $\widetilde{\mathcal{D}^0_N(\lambda) } :
C^{\infty}(G^{n+1}) \rightarrow C^{\infty}(G^{n+1})$ is the operator
induced by $\mathcal{D}_N^0(\lambda)$.

\begin{proof}[Proof of Theorem \ref{HJA}] Using Theorem \ref{thm_unikhet}, 
it follows that we need to prove that $\mathcal{E} \in
\mathcal{U}(\mathfrak{n}_{n+1}^{-})$ induces a map
$$
\Ind_{P_0^{n+1}}^{G^{n+1}}(\xi_{\lambda}) \rightarrow
\Ind_{P_0^{n}}^{G^{n+1}}(\xi_{\lambda-N})
$$
if and only if $\mathcal{E}$ satisfies the following three conditions:
\begin{equation}\label{3villkor}
\begin{split}
\left[X,\mathcal{E}\right] &= 0 \; \text{for all} \; X \in \mathfrak{m}_n, \\ 
\left[X, \mathcal{E}\right]  & \in \ARBIT \; \text{for all} \; 
X \in \mathfrak{n}_n^+,\\
\left[H, \mathcal{E}\right] &= -N \mathcal{E}.
\end{split}
\end{equation}

Observe that $\mathcal{E}$ induces a map if and only if
$$
(\mathcal{E} u)(x \tilde{m} \tilde{a} \tilde{n}) = \xi_{\lambda-N}
(\tilde{a}) ( \mathcal{E} u)(x)
$$ 
for all $\tilde{m}\tilde{a}\tilde{n} \in P_0^n$ and $u \in
\text{Ind}^{G^{n+1}}_{P_0^{n+1} }(\xi_{\lambda})$. First, we prove that
for all $u \in \text{Ind}^{G^{n+1}}_{P_0^{n+1} }(\xi_{\lambda})$, $x
\in G^{n+1}$ and $\tilde{m} \in M_0^{n}$ we have
\begin{equation}\label{STEPA}
(\mathcal{E} u)(x \tilde{m}) = ( \mathcal{E} u)(x)
\end{equation}
if and only if $[X,\mathcal{E}] = 0$ for all $X \in
\mathfrak{m}_n$. 

Assume that $[X,\mathcal{E}]=0$ for all $X \in
\mathfrak{m}_n$. Using $\exp(\mathfrak{m}_n) = M_n^0$, it follows that
$\Ad(\tilde{m}) \mathcal{E} = \mathcal{E}$ for $\tilde{m} \in
M^n_0$. Now for all $\tilde{m} \in M^{n}_0$ and all $u \in
\text{Ind}_{P_0^{n+1}}^{G^{n+1}}(\xi_\lambda)$ we get using
\eqref{ad_komm} that
\begin{equation}\label{STEPA__}
\begin{split}
(\mathcal{E} u)(x \tilde{m}) & = (R(\tilde{m}) \mathcal{E} u)(x) =
\big( \Ad(\tilde{m}) \mathcal{E} \big) u (x). \\
\end{split}
\end{equation} 

Conversely, assume that \eqref{STEPA} is true. Then
\eqref{STEPA__} shows that $\Ad(\tilde{m})\mathcal{E}u = \mathcal{E}u$
for all $u \in \text{Ind}_{P_0^{n+1}}^{G^{n+1}}(\xi_\lambda)$ and all
$\tilde{m} \in M^n_0$, and Lemma \ref{aux_lemma} shows that
$\Ad(\tilde{m}) \mathcal{E} = \mathcal{E}$ for $\tilde{m} \in M^n_0$.
Hence $[X,\mathcal{E}] = 0$ for all
$X \in \mathfrak{m}_n$.

Second, we prove that for all $u \in \text{Ind}^{G^{n+1}}_{P_0^{n+1}}
(\xi_{\lambda})$, $x \in G^{n+1}$ and $\tilde{n} \in N^{n}$ we have
\begin{equation}\label{STEPB}
(\mathcal{E} u)(x \tilde{n}) = ( \mathcal{E} u)(x)
\end{equation}
if and only if $\left[X, \mathcal{E}\right] \in \ARBIT \; \text{for
all} \; X \in \mathfrak{n}_n^+$.

Assume that $\left[X, \mathcal{E}\right] \in \ARBIT
\; \text{for all} \; X \in \mathfrak{n}_n^+$.

For all $X \in \mathfrak{n}_n^+$, $u \in
\text{Ind}_{P_0^{n+1}}^{G^{n+1}}(\xi_{\lambda})$ we get for $x \in
G^{n+1}$ that
$$
([X,\mathcal{E}]u )(x) = 0, 
$$ 
and hence using that $u(x\tilde{n}) = u(x)$ for all $\tilde{n} \in
N^n$ we get
\begin{equation*}
\begin{split}
(X \mathcal{E} u)(x) = ([X, \mathcal{E} ] u)(x) +
(\mathcal{E} X)u(x) = 0.
\end{split}
\end{equation*}

Conversely, assume that for all $u \in
 \text{Ind}_{P_0^{n+1}}^{G^{n+1}}(\xi_{\lambda})$ we have
\begin{equation}\label{nollf}
([X, \mathcal{E} ] u)(x) = 0, \,\, x \in G^{n+1}, \; X \in \mathfrak{n}_n.
\end{equation}
From Lemma \ref{formlemma} we know that $[X,\mathcal{E}] \in
(\mathcal{U}(\mathfrak{n}_{n+1}^-) \otimes \mathfrak{m}_{n+1}) \oplus
(\mathcal{U}(\mathfrak{n}_{n+1}^-) \otimes \mathfrak{a}) \oplus
\mathcal{U}(\mathfrak{n}_{n+1}^-)$. Using $u \in
\text{Ind}_{P_0^{n+1}}^{G^{n+1}}(\xi_{\lambda})$, we see that if
\begin{equation*}
[X,\mathcal{E}]u = \alpha (I_{1})^{a_1}  \dots (I_{n})^{a_{n}}
H u +  \beta (I_{1})^{a_1}  \dots (I_{n})^{a_{n}}u + \dots 
\end{equation*}
we get
\begin{equation*}
[X,\mathcal{E}]u = \alpha (I_{1})^{a_1}  \dots (I_{n})^{a_{n}}
\lambda u +  \beta (I_{1})^{a_1}  \dots (I_{n})^{a_{n}} u + \dots,
\end{equation*}
and using Lemma \ref{aux_lemma} and \eqref{nollf}
we get $\beta  = -\lambda \alpha$. We conclude that $[X,\mathcal{E}] \in
\ARBIT$.

Third we show that for all $u \in \text{Ind}^{G^{n+1}}_{P_0^{n+1}}
(\xi_{\lambda})$, $x \in G^{n+1}$, $\tilde{a} \in A$
$$
(\mathcal{E} u)(x \tilde{a} ) = \xi_{\lambda-N}(\tilde{a})(\mathcal{E} u)(x)
$$
if and only if $[H,\mathcal{E}] = -N \mathcal{E}$. Assume that $[H,\mathcal{E}] = -N \mathcal{E}$. We
get using \eqref{villkorC}
\begin{equation*}
\begin{split}
(H \mathcal{E})u(g) &= ([H, \mathcal{E} ] u)(g) + (\mathcal{E} H u)(g)
= (-N \mathcal{E} u)(g) + \lambda (\mathcal{E} u)(g) \\ & = (\lambda -
N) (\mathcal{E} u)(g).
\end{split}
\end{equation*}

Conversely, assume that $(H \mathcal{E} u)(x) = (\lambda-N)
(\mathcal{E} u)(x)$ for all $u \in
\text{Ind}^{G^{n+1}}_{P_0^{n+1}}(\xi_\lambda)$. Using
$$
H \mathcal{E} u = [H,\mathcal{E}] u + \mathcal{E} Hu = [H,\mathcal{E}] u +
\lambda \mathcal{E} u,
$$
we get
$$
([H,\mathcal{E}] + N\mathcal{E})  u = 0.
$$
Now $[H, \mathfrak{n}^{-}_{n+1}] \subset \mathfrak{n}_{n+1}^{-}$ and
hence Lemma \ref{aux_lemma} shows that
$$
[H,\mathcal{E}] + N\mathcal{E} = 0.
$$
\end{proof}

\begin{proof}[Proof of Theorem \ref{HUVA_MOD}] From \eqref{MEGENSKAP} we 
know that $M^n = \{ 1_{(n)}, J_{(n)} \} \langle w_1 , w_2 \rangle
M_0^n$, and hence $\mathcal{E} \in
\mathcal{U}(\mathfrak{n}_{n+1}^{-})$ induces a map
\begin{equation}
\Ind_{M^{n+1} A (N^+)^{n+1}}^{G^{n+1}}(\xi_{\lambda}) \rightarrow
\Ind_{M^{n} A (N^+)^{n}}^{G^{n+1}}(\xi_{\lambda-N})
\end{equation}
if and only if
$\mathcal{E}$ induces a map
$$
\Ind_{M_0^{n+1} A (N^+)^{n+1}}^{G^{n+1}}(\xi_{\lambda}) \rightarrow
\Ind_{M_0^{n} A (N^+)^{n}}^{G^{n+1}}(\xi_{\lambda-N})
$$ 
and
$$
(\mathcal{E} u)(x i(w) ) = ( \mathcal{E} u)(x), \,\, x \in G^{n+1}, \,\,
$$ 
for $u \in \text{Ind}^{G^{n+1}}_{P^{n+1}}(\xi_{\lambda})$ and $w
\in \{ w_1, w_2, J_{(n)} \}$. Here $i: G^n \rightarrow G^{n+1}$ is
given by \eqref{inkl2}.

Using Theorem \ref{HJA}, we conclude that $\mathcal{E} \in
\mathcal{U}(\mathfrak{n}_{n+1}^{-})$ induces a map
$$
\Ind_{M^{n+1} A (N^+)^{n+1}}^{G^{n+1}}(\xi_{\lambda}) \rightarrow
\Ind_{M^{n} A (N^+)^{n}}^{G^{n+1}}(\xi_{\lambda-N})
$$ 
if and only if for some $c \in \mathbb{C}$ we have $\mathcal{E} = c \mathcal{D}_N^0(\lambda)$ 
and
$$
(\mathcal{E} u)(x i(w) ) = ( \mathcal{E} u)(x), \,\, x \in G^{n+1}, \,\,
$$ 
for $u \in \text{Ind}^{G^{n+1}}_{P^{n+1}}(\xi_{\lambda})$
and $w \in \{ w_1, w_2, J_{(n)} \}$.

It follows that we need to prove that $\mathcal{E} = c
\mathcal{D}_{N}^{0}(\lambda)$ for some $c \in \mathbb{C}$ and $N$ is
even if and only if $\mathcal{E} = c \mathcal{D}_N^0(\lambda)$ for $c
\in \mathbb{C}$ and
\begin{equation}
(\mathcal{E} u)(x i(w) ) = ( \mathcal{E} u)(x), \; x \in G^{n+1}
\label{nystatement}
\end{equation}
for all $u \in
\text{Ind}^{G^{n+1}}_{P^{n+1}}(\xi_{\lambda})$ and $w \in \{ w_1, w_2,
J_{(n)} \}$.

For $u \in
\text{Ind}^{G^{n+1}}_{P^{n+1}}(\xi_{\lambda})$ and $x \in G^{n+1}$ we have 
\begin{multline}
(\Ad(i(w)) \mathcal{E}) u(x) = (R(i(w)) \circ \mathcal{E} \circ
  R(i(w)^{-1}) ) u(x) \\= (R(i(w)) \circ \mathcal{E}) u(x) = (\mathcal{E}u)(
  x i(w) ).
\label{hjekv}
\end{multline}
Using \eqref{hjekv} and Lemma \ref{aux_lemma}, we conclude that
\eqref{nystatement} is satisfied if and only if
$$
\Ad(i(w)) \mathcal{E} = \mathcal{E}
$$ 
for $w \in \{ w_1, w_2, J_{(n)} \}$. Using $\Ad(w)X = wXw^{-1}$, we see that 
\begin{equation}
\begin{split}
\Ad(i(w_1)) Q_1^- &= -Q_1^-, \\
\Ad(i(w_1)) Q_2^- &= Q_2^-,  \\
\Ad(i(w_1)) Y_j^- &= Y_j^-, \,\, \text{for} \,\, j=1, \dots, n-2,\\
\Ad(i(w_2)) Q_1^- &= Q_1^-, \\
\Ad(i(w_2)) Q_2^- &= -Q_2^-, \\
\Ad(i(w_2)) Y_j^- &= Y_j^-, \,\, \text{for} \,\, j=1, \dots, n-2,\\
\Ad(i(J_{(n)})) Q_1^- &= Q_1^-, \\ 
\Ad(i(J_{(n)})) Q_2^- &= Q_2^-, \\
\Ad(i(J_{(n)})) Y_{n-2}^- &= -Y_{n-2}^{-}, \\
\Ad(i(J_{(n)})) Y_j^- &= Y_j^- \,\,\, \text{for} \,\,\, j=1, \dots, n-3.
\end{split}
\end{equation} 
Here we used that
\begin{equation*}
Q_1^- = \begin{pmatrix} 0 & -1 & 0 & 0 & \textbf{0} \\
1 & 0 & 0 & 1 & \textbf{0} \\
0 & 0 & 0 & 0 & \textbf{0} \\
0 & 1 & 0 & 0 & \textbf{0} \\
\textbf{0}  & \textbf{0} &\textbf{0} &\textbf{0} &\textbf{0} 
\end{pmatrix},
Q_2^- = \begin{pmatrix} 0 & 0 & 1 & 0 & \textbf{0} \\
0 & 0 & 0 & 0 & \textbf{0} \\
1 & 0 & 0 & 1 & \textbf{0} \\
0 & 0 &-1 & 0 & \textbf{0} \\
\textbf{0}  & \textbf{0} & \textbf{0} & \textbf{0} & \textbf{0} 
\end{pmatrix}.
\end{equation*}
We get 
\begin{equation}
\begin{split}
\text{Ad}(i(J_{(n)})) \mathcal{D}^{0}_N(\lambda) &= (-1)^N
\mathcal{D}^{0}_N(\lambda), \\
\text{Ad}(i(w_1)) \mathcal{D}^{0}_N(\lambda) &= 
\mathcal{D}^{0}_N(\lambda), \\
\text{Ad}(i(w_2)) \mathcal{D}^{0}_N(\lambda) &= 
\mathcal{D}^{0}_N(\lambda).
\label{viktigfor}
\end{split}
\end{equation} 
This completes the proof.
\end{proof}

\begin{proof}[Proof of Theorem \ref{SIGMA}]
From \eqref{MEGENSKAP} we know that $$M^n = \{ 1_{(n)}, J_{(n)} \}
\langle w_1 , w_2 \rangle M_0^n,$$and hence $\mathcal{E} \in
\mathcal{U}(\mathfrak{n}_{n+1}^{-})$ induces a map
\begin{equation}
\Ind_{M^{n+1} A (N^+)^{n+1}}^{G^{n+1}}(\xi_{\lambda}) \rightarrow
\Ind_{M^{n} A (N^+)^{n}}^{G^{n+1}}(\xi_{\lambda-N} \otimes \sigma_{-})
\end{equation}
if and only if
$\mathcal{E}$ induces a map
$$ \Ind_{M_0^{n+1} A (N^+)^{n+1}}^{G^{n+1}}(\xi_{\lambda}) \rightarrow
\Ind_{M_0^{n} A (N^+)^{n}}^{G^{n+1}}(\xi_{\lambda-N} \otimes
\sigma_{-}\vert_{M_0^n} ) = \Ind_{M_0^{n} A
(N^+)^{n}}^{G^{n+1}}(\xi_{\lambda-N} )
$$ 
and
$$ 
(\mathcal{E} u)(x i(w) ) = \sigma_{-}(w) ( \mathcal{E} u)(x) , \,\,
x \in G^{n+1}, \,\,
$$ 
for $u \in \text{Ind}^{G^{n+1}}_{P^{n+1}}(\xi_{\lambda})$ and $w
\in \{ w_1, w_2, J_{(n)} \}$. Here $i: G^n \rightarrow G^{n+1}$ is
given by \eqref{inkl2}.

Using Theorem \ref{HJA}, we conclude that $\mathcal{E} \in
\mathcal{U}(\mathfrak{n}_{n+1}^{-})$ induces a map
$$
\Ind_{M^{n+1} A (N^+)^{n+1}}^{G^{n+1}}(\xi_{\lambda}) \rightarrow
\Ind_{M^{n} A (N^+)^{n}}^{G^{n+1}}(\xi_{\lambda-N} \otimes \sigma_{-})
$$ if and only if for some $c \in \mathbb{C}$ we have $\mathcal{E} = c
\mathcal{D}_N^0(\lambda)$ and
$$
(\mathcal{E} u)(x i(w) ) = \sigma_{-}(w) ( \mathcal{E} u)(x) , \,\, x \in G^{n+1}, \,\,
$$ 
for $u \in \text{Ind}^{G^{n+1}}_{P^{n+1}}(\xi_{\lambda})$
and $w \in \{ w_1, w_2, J_{(n)} \}$.

It follows that we need to prove that $\mathcal{E} = c
\mathcal{D}_{N}^{0}(\lambda)$ for some $c \in \mathbb{C}$ and $N$ is
odd if and only if $\mathcal{E} = c \mathcal{D}_N^0(\lambda)$ for $c
\in \mathbb{C}$ and
\begin{equation}
(\mathcal{E} u)(x i(w)) = \sigma_{-}(w) (\mathcal{E}u)(x), \; x \in
G^{n+1}, \;
\label{nystatement2}
\end{equation}
for all $u \in
\text{Ind}^{G^{n+1}}_{P^{n+1}}(\xi_{\lambda})$ and $w \in
\{ w_1, w_2, J_{(n)} \}$.

For $u \in \text{Ind}^{G^{n+1}}_{P^{n+1}}(\xi_{\lambda})$ and $x \in
G^{n+1}$ we have
\begin{multline}
(\Ad(i(w)) \mathcal{E}) u(x) = (R(i(w)) \circ \mathcal{E} \circ
R(i(w)^{-1})) u(x) \\
= (R(i(w)) \circ \mathcal{E}) u(x) =
(\mathcal{E}u)(xi(w)).
\label{hjekv2}
\end{multline}
Using \eqref{hjekv2} and Lemma \ref{aux_lemma}, we conclude that
\eqref{nystatement2} is satisfied if and only if
\begin{equation}
\begin{split}
\Ad(i(w_1)) \mathcal{E} &= \sigma_{-}(w_1) \mathcal{E} = \mathcal{E}, \\
\Ad(i(w_2)) \mathcal{E} &= \sigma_{-}(w_2) \mathcal{E} = \mathcal{E},\\
\Ad(i(J_{(n)})) \mathcal{E} &= \sigma_{-}(J_{(n)}) \mathcal{E} = -\mathcal{E}.
\end{split}
\end{equation}
Using \eqref{viktigfor}, the result follows.
\end{proof}

\begin{proof}[Proof of Theorem \ref{HUVA2}]
The proof is analogous to the proof of Theorem \ref{s7-1} using that
$\hat{M}^n = \langle w_1, w_2 \rangle M_0^n$.
\end{proof}

The differential operator families in the non-compact model are defined by
\begin{equation*}
\begin{split} 
& D_N^{\text{nc}}(\lambda) : C^{\infty}((N^-)^{n+1})_{\xi_{\lambda}}
\rightarrow C^{\infty}((N^-)^n)_{\xi_{\lambda-N} \otimes \sigma}, \\ &
D_{N}^{\text{nc}}(\lambda) = \beta^n_{\xi_{\lambda-N} \otimes \sigma}
\circ D_N(\lambda) \circ (\beta_{\xi_{\lambda} }^{n+1})^{-1},
\end{split}
\end{equation*}
where
\begin{equation}
\sigma = \begin{cases} \sigma_+ & \text{if} \,\, N \,\, \text{is even}, \\ 
\sigma_{-} & \text{if} \,\, N \,\, \text{is odd}.
         \end{cases}
\label{ALTDEF}
\end{equation}
$\beta_{\xi_{\lambda} \otimes \sigma}^{m}$ is the restriction map
$\text{Ind}_{P^m}^{G^m}(\xi_\lambda \otimes \sigma) \rightarrow
C^{\infty}((N^-)^m)_{\xi_{\lambda} \otimes \sigma} \subset
C^{\infty}(\mathbb{R}^{m-1})$ and $\sigma$ is given by \eqref{ALTDEF}.

Using the identification \eqref{IDENT}, we have
\begin{equation}\label{diffop_ickekom}
\begin{split}
D^{\text{nc}}_{2N}(\lambda) & = \sum_{j=0}^{N} a_{j}(\lambda)
(\Delta_{\mathbb{M}^{n-1}})^j i^* \left(\frac{\partial}{\partial
x_n}\right)^{2N-2j}, \\ D_{2N+1}^{\text{nc}}(\lambda) & =
\sum_{j=0}^{N} b_j(\lambda) (\Delta_{\mathbb{M}^{n-1}})^j i^*
\left(\frac{\partial}{\partial x_n}\right)^{2N-2j+1},
\end{split}
\end{equation}
where
$$
\Delta_{\mathbb{M}^{n-1}} = -\frac{\partial^2}{\partial x_1^2} + \sum_{i=2}^{n-1}
\frac{\partial^2}{\partial x_i^2}
$$ 
and $(i^* \varphi)(x') = \varphi(x',0)$.

\section{Intertwining families and asymptotics of eigenfunctions}\label{formal}

In the present section we describe an alternative construction of the
families $D_{N}^{\text{nc}}(\lambda)$ in terms of the asymptotics of
  eigenfunctions of the Laplacian on anti de Sitter space (called the
  \emph{residue family} in \cite{juhl_conform}).

Let
$$ 
\mathbb{U}_{n} = \{(x_1,\ldots,x_n) \in \mathbb{R}^n \, \vert \,x_n > 0 \}
$$
be the Lorentzian upper half-space with the Lorentzian metric 
$$
x_n^{-2} \left( -dx_1^2 + \sum_{i=2}^{n} d x_i^2 \right)
$$ 
and the Laplacian
$$
\Delta_{\mathbb{U}_n} = x_n^2 \Delta_{\mathbb{M}^{n-1} } + x_{n}^2
\frac{\partial^2}{\partial x_n^2} - (n-2)x_n \frac{\partial}{\partial x_n}.
$$
We seek \emph{formal} solutions $u \in C^\infty(\mathbb{U}_n)$ to the equation
$$
-\Delta_{\mathbb{U}_n } u = \lambda(n-1-\lambda)u
$$
in the form
$$ 
u(x) \sim \sum_{j \geq 0} x_n^{\lambda+j} c_{j}(x'), \; x=(x',x_n), \; 
c_{\text{odd}} = 0.
$$ 
The coefficients $c_j$ satisfy recursive relations such that the
coefficients $c_{2j}$, $j \geq 1$, are determined by $c_0$. More
precisely, we have maps
$$
T_{2j}(\lambda): c_0(x') \mapsto c_{2j}(\lambda,x')
$$
which are differential operators (depending on $\lambda$) of order $2j$ on
$\mathbb{R}^{n-1}$, where $T_0(\lambda) = \text{id}$.

We define $S_{2N}(\lambda): C^{\infty}(\mathbb{R}^n) \rightarrow
C^{\infty}(\mathbb{R}^{n-1})$ by
$$ 
S_{2N}(\lambda) = \sum_{j=0}^{N} \frac{1}{(2N\!-\!2j)!} T_{2j}(\lambda) i^* 
\left( \frac{\partial}{\partial x_n} \right)^{2N-2j}.
$$

\begin{theorem} The families $S_{2N}(\lambda+n-1-2N)$ and $D_{2N}^{\mathrm{nc}}(\lambda)$
coincide, up to a rational function in $\lambda$. 
\label{LIKAFAMILJ}
\end{theorem}

\begin{proof}
A computation shows that $T_{2j}(\lambda) = A_{2j}(\lambda) 
(\Delta_{\mathbb{M}^{n-1}})^j$, where
\begin{equation}\label{rekurr}
A_{2j-2}(\lambda) + 2j(2j+2\lambda+1-n) A_{2j}(\lambda) = 0
\end{equation}
and $A_0(\lambda) = 1$. We see that $S_{2N}(\lambda)$ can be written in the form
\begin{equation}
\begin{split}
&\sum_{j=0}^{N} \frac{1}{(2N\!-\!2j)!} A_{2j}(\lambda)
(\Delta_{\mathbb{M}^{n-1}} )^{j} i^* \left( \partial/\partial x_n
\right)^{2N-2j} \\ &= \sum_{j=0}^{N} B_{2j}(\lambda) (\Delta_{\mathbb{M}^{n-1}})^{j} 
i^* \left( \partial/\partial x_n \right)^{2N-2j},
\end{split}
\end{equation}
where the coefficients $B_{2j}(\lambda)$ satisfy
$$
B_{0}(\lambda) = \frac{1}{(2N)!},
$$
and
\begin{equation}\label{Brekurr}
B_{2j-2}(\lambda) + \frac{2j(2j\!+\!2\lambda\!+\!1\!-\!n)}
{(2N\!-\!2j\!+\!2)(2N\!-\!2j\!+\!1)} B_{2j}(\lambda) = 0.
\end{equation}
We recall that $D_{2N}^{\mathrm{nc}}(\lambda)$ is given by \eqref{diffop_ickekom}, 
where
$$ 
a_{j-1}(\lambda) + \frac{2j(2j\!+\!2\lambda\!+\!n\!-\!1\!-\!4N)}
{(2N\!-\!2j\!+\!2)(2N\!-\!2j\!+\!1)} a_j(\lambda) = 0.
$$
Now we observe that
$$
\frac{B_{2j}(\lambda\!+\!n\!-\!1\!-\!2N)}{a_j(\lambda)}
$$ 
does not depend on $j$. Using $a_N(\lambda)=1$, \eqref{Brekurr} shows
\begin{align*}
B_{2N}(\lambda) & = \frac{1}{(2 \cdot 4 \cdot \ldots \cdot 2N)
(n\!-\!1\!-\!2\lambda\!-\!2)(n\!-\!1\!-\!2\lambda\!-\!4) 
\dots (n\!-\!1\!-\!2\lambda\!-\!2N)} \\ 
& = \frac{1}{2^{N} N!(n\!-\!1\!-\!2\lambda\!-\!2)(n\!-\!1\!-\!2\lambda\!-\!4)\dots
(n\!-\!1\!-\!2\lambda\!-\!2N)}.
\end{align*}
We conclude that
$$ 
S_{2N}(\lambda\!+\!n\!-\!1\!-\!2N) =
B_{2N}(\lambda\!+\!n\!-\!1\!-\!2N) D_{2N}^{\text{nc}}(\lambda).
$$
\end{proof}

If we define
$$
S_{2N+1}(\lambda) = \sum_{j=0}^{N} \frac{1}{(2N\!+1-\!2j)!} T_{2j}(\lambda) i^* 
\left( \frac{\partial}{\partial x_n} \right)^{2N+1-2j},
$$
we get the following analogous result:
\begin{theorem}
The families $S_{2N+1}(\lambda+n-1-(2N+1))$ and
$D_{2N+1}^{\mathrm{nc}}(\lambda)$ coincide, up to a rational function in $\lambda$.
\end{theorem}
\begin{proof}
The proof is similar to the proof of Theorem \ref{LIKAFAMILJ}; we refer to Theorem 5.2.6 in \cite{juhl_conform}. 
\end{proof}

\section*{Appendix}

\subsection*{Summary of commutator relations for $\mathfrak{g}_n 
= \mathfrak{o}(2,n-1)$}

\begin{align*}
[Q_{1}^{+}, Y_{j}^{+} ] & = 0, & [Q_1^- , Y_j^+ ] &= Z_j^+ + Z_j^-, \\
[Q_{1}^{+}, Y_{j}^{-} ] & = -Z_{j}^+ - Z_{j}^-, & [Q_1^- , Y_j^-] &= 0,
\\ [Q_{1}^{+}, Z_{j}^{+} ] & = Y_j^+, & [Q_1^-, Z_j^+] &= -Y_j^-, \\
[Q_1^+, Z_j^- ] &= Y_j^+, & [Q_1^-, Z_j^-] & = -Y_j^-, \\ {} \\
[Q_{2}^{+}, Y_{j}^{+} ] & = 0, & [Q_2^- , Y_j^+ ] & = Z_j^+ - Z_j^-, \\
[Q_{2}^{+}, Y_{j}^{-} ] & = Z_j^+ - Z_j^-, & [Q_2^- , Y_j^-] &= 0, \\
[Q_{2}^{+}, Z_{j}^{+} ] & = Y_j^+, & [Q_2^-, Z_j^+] &= Y_j^-, \\
[Q_2^+, Z_j^- ] &= -Y_j^+ , & [Q_2^-, Z_j^-] &= -Y_j^-, \\ {} \\
[Q_1^+, Q_2^+] &= 0, & [H, Q_1^+] &= Q_1^+, \\ [Q_1^-, Q_2^-] &= 0, &
[H, Q_1^-] &= -Q_1^-, \\ [Q_1^+, Q_2^-] &= -2H_2, & [H, Q_2^+] &=
Q_2^+, \\ [Q_1^-, Q_2^+] &= 2H_2, & [H, Q_2^-] &= -Q_2^-, \\ [Q_1^+,
Q_1^-] &= 2H, & [H_2, Q_1^+] &= -Q_2^+, \\ [Q_2^+, Q_2^-] &= 2H,  &
[H_2, Q_1^-] &= Q_2^-, \\ [Q_i^{\pm}, M_{jk} ] &= 0, & [H_2, Q_2^+] &=
-Q_1^+, \\ {}&{} & [H_2, Q_2^-] &= Q_1^-,  {} \\ \\ [Y_{i}^{+},
Y_{j}^{-}] &= 2\delta_{ij} H + 2M_{ij}, & [Z_{i}^{+}, Z_{j}^{-}] &=
2\delta_{ij} H_2 + 2M_{ij}, \\ [M_{ij}, Y^{\pm}_{r}] &= \delta_{jr}
Y_{i}^{\pm} - \delta_{ir}Y_{j}^{\pm}, & [M_{ij}, Z^{\pm}_{r}] &=
\delta_{jr} Z_{i}^{\pm} - \delta_{ir}Z_{j}^{\pm}, \\ [H, Y_{j}^{\pm}]
&= \pm Y_{j}^{\pm}, & [H_2, Z_{j}^{\pm}] &= \pm Z_{j}^{\pm}, \\ [H,
M_{ij}] &= 0, & [H_2, M_{ij}] &= 0, \\ {} \\ [Y_j^+ , Z_k^+] &=
\delta_{jk} (Q_1^+ - Q_2^+), & [Y_j^+, Z_k^-] &= \delta_{jk} (Q_1^+ +
Q_2^+), \\ [Y_j^-, Z_k^+] &= \delta_{jk} (-Q_1^- - Q_2^-), & [Y_j^-,
Z_k^-] &= \delta_{jk} (-Q_1^- + Q_2^-).
\end{align*}

\bibliography{geomfin}
\bibliographystyle{plain}

\end{document}